\UseRawInputEncoding

\documentclass[final]{siamltex}
\usepackage{amsmath}
\usepackage{amsfonts}
\usepackage{mathrsfs}
\usepackage{multirow}
\usepackage{amssymb}
\usepackage{hyperref}
\usepackage{booktabs}
\usepackage{stmaryrd}
\usepackage{listings,url,verbatim}    
\usepackage{algpseudocode}
\usepackage[ruled]{algorithm}
\usepackage[dvips]{epsfig,graphicx}
\usepackage{subfigure}
\usepackage{longtable}
\usepackage{lscape}
\usepackage{rotating}
\usepackage{color}
\newtheorem{remark}{\bf Remark}

\makeatletter
\newif\if@borderstar
\def\bordermatrix{\@ifnextchar*{%
\@borderstartrue\@bordermatrix@i}{\@borderstarfalse\@bordermatrix@i*}%
}
\def\@bordermatrix@i*{\@ifnextchar[{\@bordermatrix@ii}{\@bordermatrix@ii[()]}}
\def\@bordermatrix@ii[#1]#2{%
\begingroup
\m@th\@tempdima8.75\p@\setbox\z@\vbox{%
\def\cr{\crcr\noalign{\kern 2\p@\global\let\cr\endline }}%
\ialign {$##$\hfil\kern 2\p@\kern\@tempdima & \thinspace %
\hfil $##$\hfil && \quad\hfil $##$\hfil\crcr\omit\strut %
\hfil\crcr\noalign{\kern -\baselineskip}#2\crcr\omit %
\strut\cr}}%
\setbox\tw@\vbox{\unvcopy\z@\global\setbox\@ne\lastbox}%
\setbox\tw@\hbox{\unhbox\@ne\unskip\global\setbox\@ne\lastbox}%
\setbox\tw@\hbox{%
$\kern\wd\@ne\kern -\@tempdima\left\@firstoftwo#1%
\if@borderstar\kern2pt\else\kern -\wd\@ne\fi%
\global\setbox\@ne\vbox{\box\@ne\if@borderstar\else\kern 2\p@\fi}%
\vcenter{\if@borderstar\else\kern -\ht\@ne\fi%
\unvbox\z@\kern-\if@borderstar2\fi\baselineskip}%
\if@borderstar\kern-2\@tempdima\kern2\p@\else\,\fi\right\@secondoftwo#1 $%
}\null \;\vbox{\kern\ht\@ne\box\tw@}%
\endgroup
}

\makeatletter
\newenvironment{breakablealgorithm}
{
	\begin{center}
		\refstepcounter{algorithm}
		\hrule height.8pt depth0pt \kern2pt
		\renewcommand{\caption}[2][\relax]{
			{\raggedright\textbf{\ALG@name~\thealgorithm} ##2\par}%
			\ifx\relax##1\relax 
			\addcontentsline{loa}{algorithm}{\protect\numberline{\thealgorithm}##2}%
			\else 
			\addcontentsline{loa}{algorithm}{\protect\numberline{\thealgorithm}##1}%
			\fi
			\kern2pt\hrule\kern2pt
		}
	}{
		\kern2pt\hrule\relax
	\end{center}
}
\makeatother

\begin{document}

\title{\bf Tensor Regularized Total Least Squares Method with Applications to Image and Video Deblurring}
\author{Feiyang Han \thanks{E-mail: 19110180030@fudan.edu.cn. School of Mathematical Sciences, Fudan University, Shanghai, 200433, P. R. of China. This author is supported by the National Natural Science Foundation of China under grant 12271108. } \and  Yimin Wei \thanks{Corresponding author (Y. Wei). E-mail: ymwei@fudan.edu.cn and yimin.wei@gmail.com. School of Mathematical Sciences and Key Laboratory of Mathematics for Nonlinear Sciences, Fudan University, Shanghai, 200433, P. R. China. This author is supported by Innovation Program of Shanghai Municipal Education Commission and the National Natural Science Foundation of China under grant 12271108.
}
\and
Pengpeng Xie \thanks{E-Mail: xie@ouc.edu.cn. School of Mathematical Sciences, Ocean University of China, Qingdao 266100, P.R. of 	China. This author is supported by the National Natural Science Foundation of China under grant 12271108.}
}

\maketitle

\begin{abstract}
Total least squares (TLS) is an effective method for solving linear equations with the situations, when noise is not just in observation matrices but also in mapping matrices. Moreover, the Tikhonov regularization is widely used in plenty of ill-posed problems. In this paper, we extend the regularized total least squares (RTLS) method from the matrix form due to Golub, Hansen and O'Leary,  to the tensor form proposing the tensor regularized total least squares (TR-TLS) method for solving ill-conditioned tensor systems of equations. Properties and algorithms about the solution of the TR-TLS problem, which might be similar to  those of the RTLS, are also presented and proved. Based on this method, some applications in image and video deblurring are explored. Numerical examples illustrate the TR-TLS, compared with the existing methods.
\end{abstract}



\section{Introduction}
Regularized total least squares method (RTLS) is a practical technique for solving ill-conditioned overdetermined linear problems and discrete linear ill-posed problems. This method consists of two important parts, the total least squares (TLS) and the Tikhonov regularization. Among them, the TLS method was proposed by Golub and Van Loan in 1980 \cite{golub1980analysis}. Different from the more traditional least squares (LS) problem, the TLS considers that in the linear equations $Ax\approx b$, not only the right-hand  vector $b$ is affected by the error vector $f$, but also the coefficient matrix $A$ is affected by the error matrix $E$. The LS is aimed to solve the computational problems, 
\begin{equation}
\min _{x, \varepsilon}\|\varepsilon\|_2 \quad \text { s.t. } \quad A x=b+\varepsilon,
\end{equation}
where $A \in \mathbb{R}^{m \times n}, x\in\mathbb R^{n}, b$ and $ \varepsilon \in \mathbb{R}^m$, while the TLS is designed to solve the following issues,
\begin{equation}
\min _{x, E, f}\|(E, f)\|_F \quad \text { s.t. } \quad(A+E) x=b+f,
\end{equation}
where $A, E \in \mathbb{R}^{m \times n}, x\in\mathbb R^{n}, b$ and $ f \in \mathbb{R}^m$. 

In 1987 \cite{van1987algebraic}, Van Huffel analyzed the relationships between TLS and LS solutions of  linear equations, $Ax \approx b$. There are two essential monographs \cite{ markovsky2007overview, van1991total} and plenty of papers \cite{beck2005a, fierro1997regularization, gratton2013sensitivity, zbMATH06609241, zbMATH05998011, liu2022multi, mastronardi2000fast, zbMATH01734139, zheng2017condition} to summarize the properties and varieties of TLS method. Since the TLS was proposed, it has gained wide attention and applications in  signal processing \cite{de2008blind}, data mining \cite{han2022tls, li2015solving} and  image processing \cite{benthib2022a, el2021tensor, markovsky2007overview, vasilescu2002multilinear}.

When the coefficient matrix $A$ tends to be ill-conditioned, the solution is very sensitive to perturbation. As a consequence, a regularization constraint is required for the original problem. The Tikhonov regularization is a common method to deal with ill-conditioned problems \cite{chu2011condition, schaffrin2010total, wei2016tikhonov, xiang2015randomized}. In 1999 \cite{golub1999tikhonov}, Golub, Hansen and O'Leary provided Tikhonov regularization methods to keep the solution stable for the highly ill-conditioned linear TLS problem and proposed RTLS method. The RTLS problem holds the form as follows,
\begin{equation}
\min_{\tilde{A}, \tilde{b},x} \|(A, b)-(\tilde{A}, \tilde{b})\|_F \quad \text { s.t. } \quad \tilde{b}=\tilde{A} x, \quad\|L x\|_2 \leq \delta,
\end{equation}
with its corresponding Lagrange multiplier formulation,
\begin{equation}
\hat{\mathcal{L}}(\tilde{A}, x, \mu)=\|(A, b)-(\tilde{A}, \tilde{A} x)\|_F^2+\mu\left(\|L x\|_2^2-\delta^2\right),
\end{equation}
where $\mu$ is the Lagrange multiplier. Since the RTLS  was raised, there have been plenty of results for solving the RTLS problem in the matrix form. In 2001 and 2005 respectively, Guo and Renaut \cite{guo2002regularized, renaut2004efficient} generated two different algorithms. Sima, Van Huffel and Golub \cite{sima2004regularized} presented a computational approach for solving the RTLS in 2003. Beck and Ben-Tal \cite{beck2006solution} discussed more properties in 2006. Besides, Zare and Hajarian considered the RTLS as an optimization problem and generated a Gauss-Newton algorithm in 2022 \cite{zare2022efficient}.

In data science, the term ``tensor" often refers to multidimensional arrays \cite{zbMATH06634084}. The first order tensor refers to vector (one-dimensional array), the second order tensor refers to matrix (two-dimensional array), and the third or higher order tensor refers to high-dimensional array. Tensors have an innate advantage in the regression analysis \cite{guhaniyogi2017bayesian, li2013some, li2015solving,  lock2018tensor, reichel2022weight, zhou2013tensor}. There are plenty of essential applications in tensor images and video modeling \cite{beik2021tensor, beik2020golub, benthib2022a}. On the one hand, using tensor structures to store data can preserve the spatial structure properties of higher-order data as much as possible \cite{zhou2017tensor}. On the other hand, using tensor operators to fit the behavior of the system can enhance the representation ability of the model. Miao \emph{et al.} \cite{miao2022Stochastic} introduced the tensor TLS and discussed the stochastic perturbation bounds for the tensor Moore-Penrose inverse based on the tensor-tensor product (T-product). Tensor T-product is also widely used in the fields of tensor linear systems, the tensor recovery and traffic models \cite{chen2021regularized, zbMATH07441215, ma2022randomized}. Besides, the tensor {Krylov} subspace  and {Golub}-{Kahan}-{Tikhonov} methods are used to speed up the computation in real applications \cite{guide2022tensorK, guide2022rbe, reichel2021tensor, reichel2022tensor}.
Recent results on the tensor decompositions via the tensor-tensor product can be found in \cite{Che2022fast,Che2022efficient,Chen2022tensor,Wang2020tensor}.

Based on the tensor T-product \cite{kilmer2013third, kilmer2011factorization, martin2013order}, we propose a tensor regularized total least squares (TR-TLS)  and provide the corresponding numerical algorithm. The TR-TLS has the similar form with the matrix RTLS,
\begin{equation}
\label{TRTLS_0001}
\begin{aligned}
\min_{\tilde{\mathcal A}, \tilde{\mathcal B}, \mathcal X}\left\|\left(\mathcal A,\mathcal B\right) - \left(\tilde{\mathcal A}, \tilde{\mathcal A}*_{\rm T}\mathcal X\right)\right\|_F\ \ \ \text{s.t.} \ \  \ \tilde{\mathcal B}=\tilde{\mathcal A}*_{\rm T}\mathcal X,\ \ \left\|\mathcal  K*_{\rm T} \mathcal X\right\|_F\leq\delta,
\end{aligned}
\end{equation}
where ``$*_{\rm T}$'' is the tensor T-product, which will be introduced in Section 2. In (\ref{TRTLS_0001}), $\tilde{\mathcal A}$ and $ \tilde{\mathcal B}$ are true data, while ${\mathcal A}$ and $ {\mathcal B}$ are observed data with errors. Some relative works were studied by El Guide \emph{et al.} \cite{el2021tensor} to generate the tensor regularized LS problem, using the generalized tensor Golub-Kahan and the GMRES. The tensor regularized LS problem could be transferred into ridge regression methods \cite{gazagnadou2022ridgesketch}.

This paper is organized as follows. In Section 2, some fundamental definitions, properties and notations of T-product operator are listed for convenience. At the same time, some main results are reviewed. The TR-TLS problem and some important theorems are derived in Section 3. Besides, an iterative algorithm is proposed to solve the TR-TLS problem, not only for single lateral slices but also for multi lateral slices. The numerical experiments which explore the applications in ill-posed images and video deblurring problems are designed in Section 4. Finally, in Section 5, we summarize and analyze all of the results. Furthermore, some future research directions are investigated.

\section{Notations and Preliminaries}
In this section, notations and the tensor-tensor T-product from the numerical linear algebra will be introduced. At the same time, some essential lemmas, theorems and properties \cite{hao2013facial, kilmer2013third, kilmer2011factorization, newman2018stable} will be analyzed.

\begin{definition}
Given a tensor $\mathcal A\in\mathbb R^{m\times n\times p}$. Its frontal, horizontal and lateral slices are  defined respectively by
\begin{equation*}
\begin{cases}
&\mathcal A[:,\ :,\ i]\in\mathbb R^{m\times n\times 1},\ \ i=1,2,\ldots,p\\
&\mathcal A[j,\ :,\ :]\in\mathbb R^{1\times n\times p},\ \ j=1,2,\ldots,m\\
&\mathcal A[:,\ k,\ :]\in\mathbb R^{m\times 1\times p},\ \ k=1,2,\ldots,n.
\end{cases}
\end{equation*}
\end{definition}
This can be illuminated more clearly by Fig. \ref{F21}.

\begin{figure}[ht]
\centering
\includegraphics[width=4in]{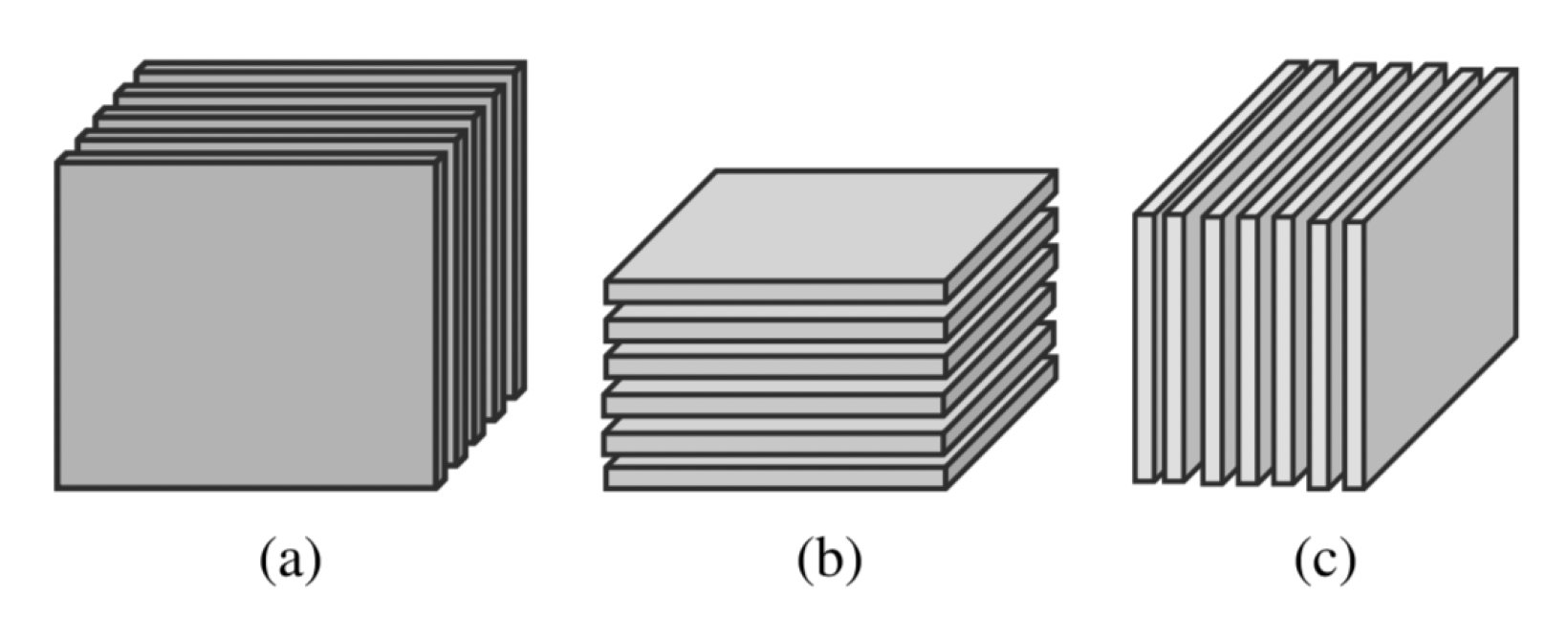}
\caption{{\rm (a)} frontal, {\rm (b)} horizontal, {\rm (c)} lateral slices of a third order tensor}
\label{F21}
\end{figure}

Taking real number field as an example, we introduce the notations used in this paper.

\begin{table}[htbp]
\centering
\caption{Some essential notations}
\begin{tabular}{cccc}
\toprule
\textbf{Item} & \textbf{Number Fields} & \textbf{Examples} & \textbf{Notations}  \\
\hline
Scalar     &      $\mathbb R^{1}$    &       $a,b,c$ & Lowercase letters  \\

 Vector    &      $\mathbb R^{n}$              &    $\textbf{a,b,c}$   &    Bold lowercase letters \\
 
  Matrix   &        $\mathbb R^{m\times n}$            & $A, B, C$    &   Capital letters \\
  Tensor & $\mathbb R^{m\times n\times p}$  & $\mathcal {A, B ,C}$&Fraktur Capital Letters\\
  Tube of Tensor &$\mathbb R^{1\times 1\times p}$ & $\tilde{ a}, \tilde{ b}, \tilde{ c}$&  Lowercase letters with tilde\\
  Slice of Tensor &$\mathbb R^{m\times 1\times p}$ & \textbf{A, B, C} & Bold capital letters\\
  \bottomrule
\end{tabular}
\end{table}
However, if we want to emphasize that an array is a tensor, we will adopt the notation associated with tensors, the fraktur capital letters.
\begin{definition}
{\rm (Tensor Block Circulant Operator)}
The definition of tensor block circulant operators can be presented as follows,
\begin{equation}
\begin{aligned}
\rm{bcirc}(\mathcal A)&=
\begin{bmatrix}
A^{(1)} & A^{(p)} & A^{(p-1)} & \cdots & A^{(2)}\\
A^{(2)} & A^{(1)} & A^{(p)} & \cdots & A^{(3)}\\
A^{(3)} & A^{(2)} & A^{(1)} & \cdots & A^{(4)}\\
\vdots & \vdots & \vdots & \ddots & \vdots\\
A^{(p)} & A^{(p-1)} & A^{(p-2)} & \cdots & A^{(1)}
\end{bmatrix}\\
&=\left(F^{\rm H}_p\otimes I_m\right)
\begin{bmatrix}
A_1 & & & \\
 & A_2 & & \\
 & & \ddots & \\
 & & & A_p
\end{bmatrix}
\left(F_p\otimes I_n\right)\in\mathbb{R}^{mp\times np},
\end{aligned}
\end{equation}
where $\mathcal A\in\mathbb{R}^{m\times n\times p}$, $A^{(i)}\in\mathbb R^{m\times n},\ A_i\in\mathbb C^{m\times n}, (i=1,2,\ldots,p)$ and $F_p=\frac{1}{\sqrt p}\left(\omega^{(i-1)(j-1)}\right)_{i=1,j=1}^p$ is the discrete Fourier matrix, where $\omega =e^{-2\pi {\bf i} /p}$ is a primitive {\rm $p$-th} root of unity in which ${\bf i}^{2}=-1$, $I_m$ is the identity matrix of order $m$.
\begin{equation*}
F_p=\frac{1}{\sqrt{p}}
\begin{bmatrix}
1 & 1 & 1 & 1 & \cdots & 1 \\
1 & \omega^{1}& \omega^{2}& \omega^{3}& \cdots & \omega^{p-1} \\
1 & \omega^{2}& \omega^{4}& \omega^{6}& \cdots & \omega^{2(p-1)} \\
1 & \omega^{3}& \omega^{6}& \omega^{9}& \cdots & \omega^{3(p-1)} \\
\vdots& \vdots& \vdots&\vdots&\ddots&\vdots\\
1 & \omega^{p-1}& \omega^{2(p-1)}& \omega^{3(p-1)}& \cdots & \omega^{(p-1)(p-1)} \\
\end{bmatrix}\in\mathbb{R}^{p\times p}.
\end{equation*}
\end{definition} 
\begin{definition} {\rm (Tensor T-product)} Suppose $\mathcal A\in\mathbb{R}^{m\times n\times p}$ and $\mathcal B\in\mathbb{R}^{n\times s\times p}$. Then the tensor product (T-product) can be defined as
\begin{equation}
\label{T_product}
\mathcal A *_{\rm T} \mathcal B = \rm{fold}({\rm bcric}(\mathcal A)\rm{unfold}(\mathcal B)),
\end{equation}
where ``fold'' is the inverse operation of 
\begin{equation*}
\rm{unfold}(\mathcal B):=\begin{bmatrix}
\mathcal B^{(1)}\\
\mathcal B^{(2)}\\
\vdots\\
\mathcal B^{(p)}\\
\end{bmatrix}\in\mathbb{R}^{np\times s}, \qquad {\mathcal B^{(i)}}\in\mathbb{R}^{n\times s}.
\end{equation*}
\end{definition}

From now, unless stated otherwise, the T-product is written in shorthand as ``$*$''. Considering the $(i,j)$-th tube of the tensor $\mathcal C\in\mathbb{R}^{1\times 1\times p}$, the T-product in (\ref{T_product}) is similar to matrix multiplications,
\begin{equation}
\mathcal C(i,j,:)=\sum_{k=1}^n\mathcal A(i,k,:) * \mathcal B(k,j,:).
\end{equation}

Using notations from MATLAB, we define  
$$
\widehat{\mathcal A}=\mbox{fft}(\mathcal A,[\ ],3).
$$
as the discrete Fast Fourier Transform (FFT) of a given tensor $\mathcal A$ along the third dimension and the T-product can be written in the FFT form, which is more economical for numerical computations. Symmetrically, the inverse Fast Fourier Transform (IFFT) can be defined. 
\begin{definition}
{\rm (Tensor T-product in FFT form)} Suppose the fast Fourier transforms of the given tensors $\mathcal A\in \mathbb{R}^{m\times n\times p}$ and $\mathcal B\in \mathbb{R}^{n\times s\times p}$ are
\begin{equation}
\widehat{\mathcal A}=\rm{fft}(\mathcal A,[\ ],3),\ \ \ 
\widehat{\mathcal B}=\rm{fft}(\mathcal B,[\ ],3).
\end{equation}
The T-product between $\mathcal A$ and $\mathcal B$ is computed by
\begin{equation}
\mathcal C(:, :, i)=(\mathcal A * \mathcal B )(:, :, i)=\rm{ifft}\left(\hat{\mathcal A}(:, :, i)\hat{\mathcal B}(:, :, i)\right),
\end{equation}
for $i=1,2,\ldots,p,$ where $\hat{\mathcal A}(:,:,i)$ and $\hat{\mathcal B}(:,:,i)$ can be seen as matrices.
\end{definition}

Combining the above analysis and definitions, we can summarize the algorithm \cite{lu2016tensor} as follows.
\vspace{1em}

\begin{breakablealgorithm}
	\caption{ T-product Algorithm based on the FFT}
	\label{alg:Framwork}
	\begin{algorithmic}[0] 
	\Require ~ $\mathcal A\in\mathbb R^{m\times n\times p}$ and $\mathcal B \in\mathbb R^{n\times s\times p}$
	
	\Ensure ~ $\mathcal C = \mathcal A*\mathcal B\in\mathbb R^{m\times s\times p}$\\
	\textbf{Step 1:} Compute $\widehat{\mathcal A}$ = fft($\mathcal A$, [ ], 3) and $\widehat{\mathcal B}$ = fft($\mathcal B$, [ ], 3)\\
	\textbf{Step 2:} Compute each frontal slice of $\mathcal C$
			$$\hat{\mathcal{C}}^{(i)}= \begin{cases}\hat{\mathcal {A}}^{(i)} \hat{\mathcal {B}}^{(i)}, & i=1,2, \ldots,\left\lceil\frac{p+1}{2}\right\rceil, \\ \operatorname{conj}\left(\hat{\mathcal {C}}^{\left(p-i+2\right)}\right), & i=\left\lceil\frac{p+1}{2}\right\rceil+1, \ldots, p .\end{cases}$$\\	
	\textbf{Step 3:} Transform $\hat{\mathcal C}$ by ifft operator, $\mathcal C$ = ifft($\widehat{\mathcal C}$, [ ], 3)
	\end{algorithmic}
\end{breakablealgorithm}

\vspace{1em}

\begin{definition}{\rm (Transpose and Conjugate Transpose) }
If $\mathcal A$ is a third order tensor, whose size is $m\times n \times p$, then the transpose $\mathcal A^\top$ could be defined from transposing all of the frontal slices and reversing the order of the transposed frontal slices from 2 to $p$. Similarly, the conjugate transpose $\mathcal A^{\rm H}$ could also be defined from conjugating all of the frontal slices and reversing the order of the transposed frontal slices from 2 to $p$. Writing these two relationships in MATLAB mathematical forms, we have
 \begin{equation*}
 \begin{cases}
 \mathcal A^\top[:, :, 1] &=  \mathcal A[:, :, 1]^\top,\\
 \mathcal A^\top[:, :, i] &=  \mathcal A[:, :, p+2-i]^\top\ \ {\rm for} \ i=2,3,\ldots,p,\\
 \mathcal A^{\rm H}[:, :, 1] &=  \mathcal A[:, :, 1]^{\rm H},\\
  \mathcal A^{\rm H}[:, :, i] &=  \mathcal A[:, :, p+2-i]^{\rm H}\ \ {\rm for} \ i=2,3,\ldots,p.
 \end{cases} 
 \end{equation*}
\end{definition} 
\begin{definition}
{\rm (Identity Tensor)} The $n\times n\times p$ identity tensor $\mathcal I_{nnp}$ is defined as a tensor whose first frontal slice is the $n \times n$ identity matrix, and whose other frontal slices are all zeros.
\end{definition}

It is easy to check that for all $\mathcal A\in\mathbb{R}^{m\times n\times p}$, $\mathcal A*\mathcal I_{nnp}=\mathcal I_{mmp} * \mathcal A=\mathcal A$. 

\begin{definition}
{\rm (Orthogonal Tensor)}
Tensor $\mathcal P$ is orthogonal if and only if $\mathcal P$ satisfies 
\begin{equation*}
\mathcal P^\top * \mathcal P = \mathcal P * \mathcal P ^\top = \mathcal I.
\end{equation*}
\end{definition}

\begin{definition}
{\rm (Tensor T-Inverse)}
The inverse tensor of frontal square tensor $\mathcal A\in\mathbb{R}^{n\times n\times p} $ can be defined as $\mathcal A^{-1}$, which satisfies 
\begin{equation*}
\mathcal A^{-1} * \mathcal A=\mathcal I_{nnp},\ \ \mathcal A * \mathcal A^{-1}=\mathcal I_{nnp}.
\end{equation*}
\end{definition}

 \begin{lemma}
 \label{TTlemma}{\rm \cite{lund2020tensor}}
  Using the above definitions of T-product and bcric operator, it holds that
  
{\rm (1)} $ {\rm bcric}(\mathcal A*\mathcal B) = {\rm bcric}(\mathcal A)\cdot {\rm bcric}(\mathcal B),$

{\rm (2)} $(\mathcal A * \mathcal B)^{\rm H}=\mathcal B^{\rm H} * \mathcal A^{\rm H}$, $(\mathcal A * \mathcal B)^\top=\mathcal B^\top * \mathcal A^\top$,

{\rm (3)} ${\rm bcric}(\mathcal A^\top)={\rm bcric}(\mathcal A)^\top$,

{\rm (4)} ${\rm bcric}(\mathcal A^{\rm H})={\rm bcric}(\mathcal A)^{\rm H}$.
 \end{lemma}

\begin{definition} {\rm \cite{miao2020generalized}
(Tensor Moore-Penrose Inverse)} Suppose $\mathcal A\in\mathbb{R}^{m\times n\times p}$. If $\mathcal B\in\mathbb{R}^{n\times m\times p}$ satisfies 
\begin{equation}
\begin{array}{r}
\mathcal{A} * \mathcal{B} * \mathcal{A}=\mathcal{A},\ \mathcal{B} * \mathcal{A} * \mathcal{B}=\mathcal{B}, \ 
(\mathcal{A} * \mathcal{B})^{\top}=\mathcal{A} * \mathcal{B},\ (\mathcal{B} * \mathcal{A})^{\top}=\mathcal{B} * \mathcal{A}.
\end{array}
\end{equation}
then $\mathcal B$ is called as the tensor Moore-Penrose inverse of $\mathcal A$, which is denoted by
\begin{equation*}
\mathcal B=\mathcal A^\dagger.
\end{equation*}
\end{definition}
The singular value decomposition (SVD) has a lot of applications in plenty of fields \cite{golub2013matrix}. We transform $\rm{bcirc}(\mathcal A)$ into the Fourier domain and take the SVD into each diagonal blocks,
\begin{equation*}
\left[\begin{array}{llll}
D_1 & & &\\
& D_2 & & \\
& & \ddots & \\
& & & D_{n_3}
\end{array}\right]=\left[\begin{array}{llll}
U_1 & & & \\
& U_2 &  & \\
& & \ddots & \\
& & & U_{n_3}
\end{array}\right]\left[\begin{array}{llll}
\Sigma_1 & & &\\
& \Sigma_2 &  &\\
& & \ddots & \\
& & & \Sigma_{n_3}
\end{array}\right]\left[\begin{array}{cccc}
V_1^\top & &  &\\
& V_2^\top  &  &\\
& &  \ddots & \\
& & & V_{n_3}^\top
\end{array}\right] \text {. }
\end{equation*}
We could gain the tensor singular value decomposition (T-SVD), which has been studied in recent years \cite{lu2016tensor, lu2019tensor, miao2021t, Newman2017Image, zhang2016exact, zhang2014novel}.
\begin{theorem}
The tensor $\mathcal A\in\mathbb{R}^{m\times n\times p}$ can be factored as 
\begin{equation}
\label{T-SVD}
\mathcal A=\mathcal U * \mathcal S * \mathcal V^\top,
\end{equation}
where $\mathcal{U}, \ \mathcal{V}$ are orthogonal $m \times m \times p$ and $n \times n \times p$ tensors, respectively, and $\mathcal{S}$ is an $m \times n \times p$ f-diagonal tensor. Here, f-diagonal means that the elements of tensor $\mathcal A$ satisfy $\mathcal A[i,i,:]\neq 0$ and $\mathcal A[i,j,:]= 0$ for $i\neq j$. The factorization (\ref{T-SVD}) is called the T-SVD (i.e., tensor SVD).
\end{theorem}

\begin{definition} {\rm (vec operation of Tensor)} The ``{\rm vec}" operation of tensor $\mathcal A\in\mathbb{R}^{m\times n\times p}$ concatenates the first column through the last column in each frontal slice of the tensor and combines all of the slices together, i.e., 
\begin{equation*}
{\rm vec}({\mathcal A})=\left(\mathcal A[:, 1, 1]^\top,\ldots,{\mathcal A}[:, n, 1]^\top,\ldots, {\mathcal A}[:, 1, p]^\top,\ldots,{\mathcal A}[:, n, p]^\top\right)^\top.
\end{equation*}
\end{definition}

\begin{definition} {\rm \cite{alaa2021on}} {\rm (Tensor Kronecker Product based on T-product)} Suppose $\mathcal A\in\mathbb{R}^{m\times n\times p}$ and $\mathcal B\in\mathbb{R}^{t\times s\times p}$, whose decompositions are
\begin{equation*}
\begin{aligned}
\mathcal A &= {\rm bcirc}^{-1}\left(
\left(F^{\rm H}_p\otimes I_m\right)
\begin{bmatrix}
A_1 & & & \\
 & A_2 & & \\
 & & \ddots & \\
 & & & A_p
\end{bmatrix}
\left(F_p\otimes I_n\right)
\right),\\
\mathcal B &= {\rm bcirc}^{-1}\left(
\left(F^{\rm H}_p\otimes I_t\right)
\begin{bmatrix}
B_1 & & & \\
 & B_2 & & \\
 & & \ddots & \\
 & & & B_p
\end{bmatrix}
\left(F_p\otimes I_s\right)
\right).
\end{aligned}
\end{equation*}
Their Kronecker product based on T-product can be defined as
\begin{equation}
\mathcal A \otimes\mathcal B= {\rm bcirc}^{-1}\left(
\left(F^{\rm H}_p\otimes I_{mt}\right)
\begin{bmatrix}
A_1\otimes B_1 & & & \\
 & A_2\otimes B_2 & & \\
 & & \ddots & \\
 & & & A_p\otimes B_p
\end{bmatrix}
\left(F_p\otimes I_{sn}\right)
\right).
\end{equation}
\end{definition}

\begin{definition}
\label{T_Fnorm}
The Frobenius norm of a tensor $\mathcal X\in\mathbb{R}^{m\times n\times p}$ is
\begin{equation*}
\|\mathcal X\|_F= \sqrt{ \sum_{i=1}^m\sum_{j=1}^n\sum_{k=1}^px_{ijk}^2}.
\end{equation*}
\end{definition}
\begin{lemma}
For the T-product of two tensors $\mathcal A\in\mathbb{R}^{m\times n\times p}, \mathcal X\in\mathbb{R}^{n\times s\times p}$, the half of Frobenius norm of $\mathcal A*\mathcal X$ is noted as
\begin{equation*}
\frac{1}{2}\left\|\mathcal A*\mathcal X\right\|_F=:\mathcal F(\mathcal X).
\end{equation*}
Then the partial derivative of $\mathcal F$ with respect to $\mathcal X$ is
\begin{equation*}
\frac{\partial \mathcal F}{\partial \mathcal X}=\mathcal A^\top*\mathcal A*\mathcal X.
\end{equation*}
\end{lemma}
\begin{proof}
From the definition of the tensor-tensor T-product,
\begin{equation*}
\begin{aligned}
{\rm vec}(\mathcal A*\mathcal X)=\left[{\rm bcric}(\mathcal A)\otimes I_s\right]\cdot {\rm vec}(\mathcal X).
\end{aligned}
\end{equation*}
Note that
$$
A = {\rm bcric}(\mathcal A)\otimes I_s,\ X={\rm vec}(\mathcal X),
$$
then we have
\begin{equation*}
\hat{\mathcal F}( X)=\mathcal F(\mathcal X)=\frac{1}{2}\left\| AX\right\|_2^2.
\end{equation*}
Combining with the above equations, we obtain
\begin{equation*}
\begin{aligned}
&~~~~\frac{\partial {\mathcal F}}{\partial\ {\rm vec}(\mathcal X)}
=\frac{\partial \hat{\mathcal F}}{\partial  X}\\
&=A^\top A X\\
&=
\left[
\begin{matrix}
{\rm bcric}(\mathcal A)^\top {\rm bcric}(\mathcal A) & &\\
&\ddots&\\
&&{\rm bcric}(\mathcal A)^\top {\rm bcric}(\mathcal A) 
\end{matrix}
\right]{\rm vec}(\mathcal X)\\
&=
\left[
\begin{matrix}
{\rm bcric}(\mathcal A^\top*\mathcal A)&&\\
&\ddots&\\
&&{\rm bcric}(\mathcal A^\top*\mathcal A)
\end{matrix}
\right]{\rm vec}(\mathcal X)\\
&={\rm vec}(\mathcal A^\top *\mathcal A *\mathcal X).
\end{aligned}
\end{equation*}
Take the ${\rm vec}^{-1}$ operation on both sides of the above equations,
\begin{equation*}
\frac{\partial \mathcal F}{\partial \mathcal X}={\rm vec}^{-1}\left(\frac{\partial {\mathcal F}}{\partial\ {\rm vec}(\mathcal X)}
\right)=\mathcal A^\top *\mathcal A *\mathcal X.
\end{equation*}
Now the whole proof has been finished.
\end{proof}
\begin{remark}
This property 
provides us with a lot of convenience in the following analysis and derivation.
\end{remark}

Here we introduce a theorem that will be important for the rest of our analysis to show that the tensor T-product has the same block multiplication property as traditional matrix multiplications.
\begin{theorem}  {\rm \cite{miao2020generalized}
\label{Block Multiplication}
(Tensor Block Multiplication based on T-product)} Suppose that
$$\begin{cases}
&\mathcal{A}_1 \in \mathbb{C}^{n_1 \times m_1 \times p}, \quad 
\mathcal{B}_1 \in \mathbb{C}^{n_1 \times m_2 \times p}, \quad  \mathcal{C}_1 \in \mathbb{C}^{n_2 \times m_1 \times p}, \quad  \mathcal{D}_1 \in \mathbb{C}_1^{n_2 \times m_2 \times p},\\ &\mathcal{A}_2 \in \mathbb{C}^{m_1 \times r_1 \times p},\quad  \mathcal{B}_2 \in \mathbb{C}^{m_1 \times r_2 \times p},\quad 
\mathcal{C}_2 \in \mathbb{C}^{m_2 \times r_1 \times p},\quad 
 \mathcal{D}_2 \in \mathbb{C}_1^{m_2 \times r_2 \times p},
\end{cases}$$
then the tensor block multiplication based on T-product shares the similar form with traditional matrix multiplications,
\begin{equation}
\left[\begin{array}{ll}
\mathcal{A}_1 & \mathcal{B}_1 \\
\mathcal{C}_1 & \mathcal{D}_1
\end{array}\right] *\left[\begin{array}{cc}
\mathcal{A}_2 & \mathcal{B}_2 \\
\mathcal{C}_2 & \mathcal{D}_2
\end{array}\right]=\left[\begin{array}{ll}
\mathcal{A}_1 * \mathcal{A}_2+\mathcal{B}_1 * \mathcal{C}_2 & \mathcal{A}_1 * \mathcal{B}_2+\mathcal{B}_1 * \mathcal{D}_2 \\
\mathcal{C}_1 * \mathcal{A}_2+\mathcal{D}_1 * \mathcal{C}_2 & \mathcal{C}_1 * \mathcal{B}_2+\mathcal{D}_1 * \mathcal{D}_2
\end{array}\right] .
\end{equation}
\end{theorem}

%
%
%
%

\section{Tensor Regularized TLS methods}

Golub and Van Loan \cite{golub1980analysis} took errors of $A$ into consideration and established the total least squares for matrix linear equations. It is aimed to solve the following problem,
\begin{equation*}
\begin{aligned}
\min_{E,r}\ \left\|\left[E\ r\right]\right\|_F \ \ \ \ {\rm s.t.} \ \  \ b+r\in\mathcal R(A+E).
\end{aligned}
\end{equation*}

The Tikhonov method \cite{golub1999tikhonov} is to solve the above question takes the form
\begin{equation*}
\begin{aligned}
\min\left\|\left(A,b\right) - \left(\tilde{A}, \tilde{A}x\right)\right\|_F\ \ \ \ \ {\rm s.t.} \ \  \   \tilde{b}=\tilde{A}x,\ \ \left\|Lx\right\|_2\leq\delta.
\end{aligned}
\end{equation*}
Using the Lagrange multiplier formulation, the Lagrange function is
\begin{equation*}
\hat{\mathcal L}(\tilde{A},x,\mu)=\left\|\left(A,b\right) - \left(\tilde{A}, \tilde{A}x\right)\right\|_F^2+ \mu(\left\|Lx\right\|_2^2-\delta^2),
\end{equation*}
where $\mu$ is the Lagrange multiplier and $\mu=0$ holds if the inequality constraint is inactive.

Let us review the  matrix result of the RTLS, which was first derived by Golub, Hessen and O'Leary in 1999.
\begin{theorem}{\rm \cite[Theorem 2.1]{golub1999tikhonov} }
\label{thm3-1}
 For the RTLS problem, if the constraint is active, the solution $x^*$ satisfies
$$
\left(A^\top A+{\lambda_I} I+\lambda_L L^\top L\right) x^*=A^\top b, 
$$
where the parameters are given by
$$
\begin{cases}
{\lambda_I} =-\frac{\left\|A x^*-b\right\|_2^2}{1+\left\|x^*\right\|_2^2}, \\
\lambda_L =\mu\left(1+\left\|x^*\right\|_2^2\right), \\
\mu =-\frac{1}{\delta^2}\left(\frac{b^\top\left(A x^*-b\right)}{1+\left\|x^*\right\|_2^2}+\frac{\left\|A x^*-b\right\|_2^2}{\left(1+\left\|x^*\right\|_2^2\right)^2}\right) .
\end{cases}
$$
\end{theorem}
In 2002, Guo and Renaut  presented an iterative algorithm based on  the following theorem.
\begin{theorem}{\rm \cite{guo2002regularized}}
\label{matrix_form_diedai}
If the constraint is active, the solution $x^*$ of the RTLS problem satisfies
$$
B\left(x^*\right)\left(\begin{array}{c}
x^* \\
-1
\end{array}\right)=-{\lambda_I}\left(\begin{array}{c}
x^* \\
-1
\end{array}\right),
$$
where
$$
B\left(x^*\right)=\left(\begin{array}{cc}
A^\top A+\lambda_L\left(x^*\right) L^\top L & A^\top b \\
b^\top A & -\lambda_L\left(x^*\right) \delta^2+b^\top b
\end{array}\right)
$$
and $\lambda_L$ is determined in Theorem \ref{thm3-1}.
\end{theorem}

From now on, we will extend the above two theorems to the tensor form.
\subsection{Case I: Single Lateral Slices}

As for the tensor,  the tensor total least squares (TTLS) is aimed to solve the tensor  equations based on the T-product \cite{miao2022Stochastic},
\begin{equation}
\mathcal A *\mathcal X \approx \mathcal B,
\end{equation}
where $\mathcal A\in\mathbb R^{m\times n\times p}$, $\mathcal X\in\mathbb R^{n\times 1\times p}$ and $\mathcal B\in\mathbb R^{m\times 1\times p}$. Similar to the matrix scenarios, the regularized TTLS (R-TTLS) problem has the form,
\begin{equation}
\label{RTTLS-1}
\begin{aligned}
\min_{\tilde{\mathcal A}, \tilde{\mathcal B}, \mathcal X}\left\|\left(\mathcal A,\mathcal B\right) - \left(\tilde{\mathcal A}, \tilde{\mathcal A}*\mathcal X\right)\right\|_F\ \ \  {\rm s.t.} \  \ \tilde{\mathcal B}=\tilde{\mathcal A}*\mathcal X,\ \ \left\|\mathcal  K* \mathcal X\right\|_F\leq\delta.
\end{aligned}
\end{equation}
The corresponding Lagrange multiplier formulation is 
\begin{equation}
\label{RTTLS-2}
\hat{\mathcal L}(\tilde{\mathcal A},\mathcal X,\mu)=\left\|\left(\mathcal A,\mathcal B\right) - \left(\tilde{\mathcal A}, \tilde{\mathcal A}*\mathcal X\right)\right\|_F^2+ \mu(\left\|\mathcal K*\mathcal X\right\|_F^2-\delta^2),
\end{equation}
where $\mu$ is the Lagrange multiplier.

Let us introduce an important lemma, which will be useful for later analysis.

\begin{lemma}
\label{Commutative}
Whenever it is assumed that $\mathcal A\in \mathbb{R}^{m\times 1\times p}$ and $\mathcal B\in \mathbb{R}^{1\times1\times p}$, the T-product between $\mathcal A$ and $\mathcal B$ satisfies an equation, 
\begin{equation*}
\mathcal A * \mathcal B = (\mathcal B \otimes\mathcal I_{mmp}) * \mathcal A,
\end{equation*}
where $\mathcal I_{mmp} \in\mathbb{R}^{m\times m \times p}$.
\end{lemma}

\begin{proof}
\begin{equation*}
\begin{aligned}
&\text{bcirc}(\mathcal A*\mathcal B)=\text{bcirc}(\mathcal A)\text{bcirc}(\mathcal B)\\
=&\left(F^{\rm H}_p\otimes I_m\right)
\begin{bmatrix}
A_1 & & & \\
 & A_2 & & \\
 & & \ddots & \\
 & & & A_p
\end{bmatrix}
\left(F_p\otimes I_1\right) 
\left(F^{\rm H}_p\otimes I_1\right)
\begin{bmatrix}
b_1 & & & \\
 & b_2 & & \\
 & & \ddots & \\
 & & & b_p
\end{bmatrix}
\left(F_p\otimes I_1\right)\\
=&\left(F^{\rm H}_p\otimes I_m\right)
\begin{bmatrix}
A_1 & & & \\
 & A_2 & & \\
 & & \ddots & \\
 & & & A_p
\end{bmatrix}
\begin{bmatrix}
b_1 & & & \\
 & b_2 & & \\
 & & \ddots & \\
 & & & b_p
\end{bmatrix}
\left(F_p\otimes I_1\right)\\
=&\left(F^{\rm H}_p\otimes I_m\right)
\begin{bmatrix}
b_1A_1 & & & \\
 & b_2A_2 & & \\
 & & \ddots & \\
 & & & b_pA_p
\end{bmatrix}
\left(F_p\otimes I_1\right)\\
=&\left(F^{\rm H}_p\otimes I_m\right)
\begin{bmatrix}
b_1I_m & & & \\
 & b_2I_m & & \\
 & & \ddots & \\
 & & & b_pI_m
\end{bmatrix}
\begin{bmatrix}
A_1 & & & \\
 & A_2 & & \\
 & & \ddots & \\
 & & & A_p
\end{bmatrix}
\left(F_p\otimes I_1\right)\\
=&\left(F^{\rm H}_p\otimes I_m\right)
\begin{bmatrix}
b_1I_m & & & \\
 & b_2I_m & & \\
 & & \ddots & \\
 & & & b_pI_m
\end{bmatrix}
\left(F_p\otimes I_m\right) 
 \left(F^{\rm H}_p\otimes I_m\right)
\begin{bmatrix}
A_1 & & & \\
 & A_2 & & \\
 & & \ddots & \\
 & & & A_p
\end{bmatrix}
\left(F_p\otimes I_1\right)\\
=&{\rm bcric}(\mathcal B \otimes\mathcal I_{mmp}) {\rm bcric}(\mathcal A)\\
=&{\rm bcric}\left[(\mathcal B \otimes\mathcal I_{mmp}) *\mathcal A\right].
\end{aligned}
\end{equation*}
\end{proof}

\begin{theorem}
When the inequality constraint in (\ref{RTTLS-1}) degenerates into an equality, the TR-TLS solution $x_{\delta}$ is a solution to the following tensor equations
\begin{equation}\label{THM2_1}
\begin{aligned}
\Big[\mathcal A^\top*\mathcal A+\widetilde{\lambda_I}\otimes\mathcal I_{nnp}+\mathcal K^\top *\mathcal K * (\widetilde{\lambda_K} \otimes\mathcal I_{nnp})\Big]*\mathcal X
=\mathcal A^\top*\mathcal B.
\end{aligned}
\end{equation}

In the above equation (\ref{THM2_1}), there are two tube parameters $\widetilde{\lambda_I}\in\mathbb{R}^{1\times1\times p}$ and $\widetilde{\lambda_K}\in\mathbb{R}^{1\times1\times p}$, 
\begin{equation}
\label{para_relation}
\begin{cases}
 {\widetilde{\lambda_I}}=\mu\mathcal X^\top*\mathcal K^\top * \mathcal K * \mathcal X
 -\textbf{R}^\top * \textbf{R} * \mathcal X^\top *\mathcal X -\textbf{R}^\top *\mathcal B,\\
 \widetilde{\lambda_K}=\mu(\mathcal I_{11p}+\mathcal X^\top *\mathcal X),\\
 \mathcal B^\top * \textbf{R}-\textbf{R}^\top*\textbf{R} = \mu\mathcal X^\top*\mathcal K^\top* \mathcal K * \mathcal X,
\end{cases}
\end{equation}
when the system reaches a steady state.

Moreover, 
\begin{equation}
\textbf{R}=({\mathcal B-\mathcal A * \mathcal X})*{(\mathcal I+\mathcal X^\top*\mathcal X)}^{-1},
\end{equation}
and
\begin{equation}
\label{relation}
\widetilde{\lambda_K}*\mathcal X^\top * \mathcal K^\top*\mathcal K*\mathcal X=\mathcal B^\top * (\mathcal B-\mathcal A*\mathcal X)+\widetilde{\lambda_I}.
\end{equation}
\end{theorem}

\begin{proof}
Take the partial derivative of (\ref{RTTLS-2}) and set them to zeros,
\begin{equation}
\label{partial1}
\begin{aligned}
\frac{\partial \hat{\mathcal L}}{\partial \tilde{\mathcal A}}
=&\tilde{\mathcal A}-\mathcal A -(\mathcal B-\tilde{\mathcal A}*\mathcal X)\mathcal * \mathcal X^\top\\
=&\tilde{\mathcal A}-\mathcal A -\mathcal B * \mathcal X^\top+\tilde{\mathcal A}*\mathcal X*\mathcal X^\top= \mathcal O,
\end{aligned}
\end{equation}
and
\begin{equation}
\label{partial2}
\begin{aligned}
\frac{\partial \hat{\mathcal L}}{\partial {\mathcal X}}=-\tilde{\mathcal A}^\top * (\mathcal B-\tilde{\mathcal A}*\mathcal X)+\mu\mathcal K^\top*
\mathcal K*\mathcal X=0.
\end{aligned}
\end{equation}
It is obvious that (\ref{partial2}) is similar to the norm equation of the tensor least squares problem,
\begin{equation}
\left(\tilde{\mathcal A}^\top*\tilde{\mathcal A}+ \mu\mathcal{K}^\top*\mathcal K\right)*\mathcal X=\tilde{\mathcal A}^\top*\mathcal B.
\end{equation}
Combining (\ref{partial1}) and noting $\textbf{R} $ as $\textbf{R} := \mathcal B - \tilde{\mathcal A} * \mathcal X$, we arrive at
\begin{equation}
\begin{aligned}
&\mathcal A^\top*\mathcal A\\
=&\left(\tilde{\mathcal A}^\top-\mathcal X* \textbf{R}^\top\right) *  \left(\tilde{\mathcal A}-\textbf{R}*\mathcal X ^\top\right).
\end{aligned}
\end{equation}
Multiply both sides of the above equation by $\mathcal X$,
\begin{equation}
\label{tuidao1}
\begin{aligned}
&\mathcal A^\top*\mathcal A*\mathcal X\\
=&\tilde{\mathcal A}^\top*\tilde{\mathcal A} * \mathcal X- \mu\mathcal X*\mathcal X^\top* \mathcal K^\top * \mathcal K * \mathcal X+ \mathcal X * \textbf{R}^\top * \textbf{R} * \mathcal X^\top *\mathcal X \\
\ \ \ \ \  \ \ &- \mu \mathcal K^\top * \mathcal K * \mathcal X * \mathcal X^\top*\mathcal X\\
=&\tilde{\mathcal A}^\top*{\mathcal B} -\mu\mathcal K^\top*\mathcal K *\mathcal X
 -\mu\mathcal X*\mathcal X^\top* \mathcal K^\top * \mathcal K * \mathcal X\\
 &+ \mathcal X * \textbf{R}^\top * \textbf{R} * \mathcal X^\top *\mathcal X - \mu \mathcal K^\top * \mathcal K * \mathcal X *\mathcal X^\top*\mathcal X\\
=&\mathcal A^\top*{\mathcal B}+\mathcal X * \textbf{R}^\top*{\mathcal B} -\mu\mathcal K^\top*\mathcal K *\mathcal X - \mu\mathcal X*\mathcal X^\top* \mathcal K^\top * \mathcal K * \mathcal X\\
&+ \mathcal X * \textbf{R}^\top * \textbf{R} * \mathcal X^\top *\mathcal X - \mu \mathcal K^\top * \mathcal K * \mathcal X *\mathcal X^\top*\mathcal X,
\end{aligned}
\end{equation}
which combines $\tilde{\mathcal A} * \mathcal X=\tilde{\mathcal B}$ and (\ref{partial1}). As a direct result of
\begin{equation*}
\begin{aligned}
\textbf{R} &\in \mathbb{R}^{m\times1\times p},\ \ 
\mathcal B& \in\mathbb{R}^{m\times1\times p},
\end{aligned}
\end{equation*}
$\textbf{R}^\top*\mathcal B$ is a tube, i.e.,
\begin{equation*}
\textbf{R}^\top*\mathcal B\in\mathbb{R}^{1\times1\times p}.
\end{equation*}
In a similar way, we obtain
\begin{equation*}
\begin{cases}
\mathcal X^\top*\mathcal K^\top * \mathcal K * \mathcal X &\in\mathbb{R}^{1\times 1\times p}\\
 \textbf{R}^\top * \textbf{R} * \mathcal X^\top *\mathcal X 
 &\in\mathbb{R}^{1\times 1\times p}\\
 \mathcal X^\top * \mathcal X&\in\mathbb{R}^{1\times 1\times p}.
\end{cases}
\end{equation*}
In combination with {Lemma \ref{Commutative}}, (\ref{tuidao1}) is transformed to 
\begin{equation}
\begin{aligned}
\left(\mathcal A^\top*\mathcal A+\widetilde{\lambda_I}\otimes\mathcal I_{nnp}+\mathcal K^\top *\mathcal K * \widetilde{\lambda_K} \otimes\mathcal I_{nnp}\right)*\mathcal X
=\mathcal A^\top*\mathcal B,
\end{aligned}
\end{equation}
where
\begin{equation}
\label{lambda_I1}
\widetilde{\lambda_I}=\mu\mathcal X^\top*\mathcal K^\top * \mathcal K * \mathcal X
-\textbf{R}^\top * \textbf{R} * \mathcal X^\top *\mathcal X -\textbf{R}^\top *\mathcal B
\end{equation}
\begin{equation}
\label{lambda_2}
\widetilde{\lambda_K}=\mu(\mathcal I_{11p} + \mathcal X^\top*\mathcal X).
\end{equation}
Consider $\mu$ from (\ref{partial2}),
\begin{equation}
\mu\mathcal X^\top*\mathcal L^\top* \mathcal L * \mathcal X=\mathcal X^\top* \tilde{\mathcal A}^\top*\textbf{R}=\mathcal B^\top * \textbf{R}-\textbf{R}^\top*\textbf{R}.
\end{equation}

Notice that
\begin{equation}
\begin{aligned}
\textbf{R}&=\mathcal B-\tilde{\mathcal A}*\mathcal X\
=\mathcal B-\left({\mathcal A}+\textbf{R}*\mathcal X^\top\right)*\mathcal X\\
&=\mathcal B-\mathcal A*\mathcal X-\textbf{R}* \mathcal X^\top* \mathcal X,
\end{aligned}
\end{equation}
and we can obtain the relation
\begin{equation*}
\textbf{R}=({\mathcal B-\mathcal A * \mathcal X})*{(\mathcal I+\mathcal X^\top*\mathcal X)}^{-1}.
\end{equation*}
Finally, the difference between (\ref{lambda_2}) $*$ $\mathcal X^\top*\mathcal K^\top*\mathcal K*\mathcal X$ and  (\ref{lambda_I1}) can reach the relationship (\ref{relation})
\begin{equation*}
\begin{aligned}
&\widetilde{\lambda_K}*\mathcal X^\top*\mathcal K^\top*\mathcal K*\mathcal X-\widetilde{\lambda_I}\\
=&\textbf{R} ^\top *\textbf{R} *\mathcal X^\top*\mathcal X+\textbf{R}^\top*\mathcal B+\mu*\mathcal X^\top*\mathcal K^\top*\mathcal K*\mathcal X*\mathcal X ^\top*\mathcal X\\
=&\textbf{R} ^\top *\textbf{R} *\mathcal X^\top*\mathcal X+\mathcal B^\top*\textbf{R}+\mu*\mathcal X^\top*\mathcal K^\top*\mathcal K*\mathcal X*\mathcal X ^\top*\mathcal X\\
=&\textbf{R} ^\top *\textbf{R} *\mathcal X^\top*\mathcal X+\mathcal B^\top*(\mathcal B-\tilde{\mathcal A}*\mathcal X)+\mu*\mathcal X^\top*\mathcal K^\top*\mathcal K*\mathcal X*\mathcal X ^\top*\mathcal X\\
=&\textbf{R} ^\top *\textbf{R} *\mathcal X^\top*\mathcal X+\mathcal B^\top*(\mathcal B-{\mathcal A}*\mathcal X) +\mu*\mathcal X^\top*\mathcal K^\top*\mathcal K*\mathcal X*\mathcal X ^\top*\mathcal X+\mathcal B^\top*( \mathcal A-\tilde{\mathcal A})*\mathcal X.
\end{aligned}
\end{equation*}
Combining with (\ref{partial1}), all of the theorem have been proved.
\end{proof}

\begin{theorem}
If the constraint $\|\mathcal K*\mathcal X\|_F\leq \delta$ is active, the solution $\mathcal X^*$ of the TR-TLS satisfies,
\begin{equation}
\Psi(\mathcal X^*) *
\left(
\begin{matrix}
\mathcal X^*\\
-\mathcal I_{11p}
\end{matrix}
\right)=-
\left(
\begin{matrix}
\mathcal X^*\\
-\mathcal I_{11p}
\end{matrix}
\right)*\widetilde{\lambda_I},
\end{equation}
where
\begin{equation}
\begin{aligned}
&\Psi(\mathcal X^*)=
\left(
\begin{matrix}
\mathcal A^\top * \mathcal A+\mathcal K^\top *\mathcal K * \left(\widetilde{\lambda_K}\otimes\mathcal I_{nnp}\right) & \mathcal A^\top *\mathcal B\\
\mathcal B^\top *\mathcal A & 
-\widetilde{\lambda_K}*\mathcal X^\top*\mathcal K^\top*\mathcal K*\mathcal X +\mathcal B^\top *\mathcal B   
\end{matrix}
\right)
\end{aligned}
\end{equation}
where  $\widetilde{\lambda_I}$ and $ \widetilde{\lambda_K}$ are determined in (\ref{para_relation}).
\end{theorem}
\begin{proof}
From two equations, (\ref{THM2_1}) and (\ref{relation}), they could be transformed into,
\begin{equation*}
\left(\mathcal A^\top*\mathcal A+\mathcal K^\top *\mathcal K * \left(\widetilde{\lambda_K} \otimes\mathcal I_{nnp}\right)\right)*\mathcal X
-\mathcal A^\top*\mathcal B=-\left(\widetilde{\lambda_I}\otimes\mathcal I_{nnp}\right) * \mathcal X=-\mathcal X*\widetilde{\lambda_I},
\end{equation*}
and 
\begin{equation*}
\mathcal B^\top *\mathcal A*\mathcal X+\widetilde{\lambda_K}*\mathcal X^\top*\mathcal K^\top*\mathcal K*\mathcal X -\mathcal B^\top *\mathcal B   = \widetilde{\lambda_I}.
\end{equation*}
Together with Theorem \ref{Block Multiplication}, our proof has been completed.
\end{proof}
\newpage

\begin{breakablealgorithm}
\caption{Iterative Algorithm for the TR-TLS in Tensor Form }
	\label{Tensor Algorithm}
	\begin{algorithmic}[0] 
	\Require ~ $\mathcal A\in\mathbb R^{m\times n\times p}$, $\mathcal B \in\mathbb R^{m\times 1\times p}$, initial value  $\mathcal X^{(0)}\in\mathbb R^{n\times 1\times p}$, error bound $\delta_0$, maximum of iteration steps $K_{\max}$
	
	\Ensure ~ Final value of $\mathcal X\in\mathbb R^{n\times 1\times p}$\\
	\textbf{For} $k=0,1,2,\ldots,K_{\max}$
	
	\textbf{Step 1.} Calculate $\Psi(\mathcal X^{(k)})$
	
	\textbf{Step 2.} Compute 
	$$
	\hat{\mathcal X}^{(k+1)}= \Psi(\mathcal X^{(k)}) * \left(
	\begin{matrix}
	\mathcal X^{(k)}\\
	-\mathcal I_{11p}
	\end{matrix}
	\right)
	$$
	
	\textbf{Step 3.} Normalize the $\mathcal X^{(k+1)}$,
	$$
	\mathcal X^{(k+1)}=\hat{\mathcal X}^{(k+1)}(1:n,1,:) / \hat{\mathcal X}^{(k+1)}(n+1, 1, 1)
	$$
	
	\textbf{Step 4.}
	\textbf{If} $$\frac{\|\mathcal X^{(k+1)} - \mathcal X^{(k)}\|_F}{\|\mathcal X^{(k)}\|_F}\leq \delta_0$$
	
	\textbf{Break}
	\\
	\textbf{End For.}
	\end{algorithmic}
\end{breakablealgorithm}

\vspace{1em}
Moreover, a theorem in the matrix form is introduced.

\begin{theorem}
If the constraint $\|\mathcal K*\mathcal X\|_F\leq \delta$ is active, the solution $\mathcal X^*$ of the TR-TLS satisfies,
\begin{equation*}
\Gamma(\mathcal X^*)
\left(
\begin{matrix}
\rm{unfold}(\mathcal X)\\
-\rm{unfold}(\mathcal I_{11p})
\end{matrix}
\right)
=-\Lambda
\left(
\begin{matrix}
\rm{unfold}(\mathcal X)\\
-\rm{unfold}(\mathcal I_{11p})
\end{matrix}
\right),
\end{equation*}
where
\begin{equation*}
\Gamma(\mathcal X^*)=
\left(
\begin{matrix}
M & N\\
N^\top & P
\end{matrix}
\right), \qquad 
\Lambda=
\left(
\begin{matrix}
\Lambda_1 & 0\\
0 & \Lambda_2
\end{matrix}
\right),
\end{equation*}
and
\begin{equation*}
\begin{cases}
M&={\rm bcric}\left(\mathcal A^\top * \mathcal A+\mathcal K^\top *\mathcal K * (\widetilde{\lambda_K}(\mathcal X^*) \otimes\mathcal I_{nnp}) \right)\\
N&={\rm bcric}\left(\mathcal A^\top *\mathcal B\right)\\
P&={\rm bcric}\left(-\widetilde{\lambda_K}(\mathcal X^*)*\mathcal X^\top*\mathcal K^\top*\mathcal K*\mathcal X +\mathcal B^\top *\mathcal B \right)\\
\Lambda_1&={\rm bcric}\left(\widetilde{\lambda_I}\otimes \mathcal I_{nnp}\right)\\
\Lambda_2&={\rm bcric}\left(\widetilde{\lambda_I}\right).
\end{cases}
\end{equation*}
\end{theorem}

\begin{proof}
Taking the ``unfold'' operator on the equations (\ref{THM2_1}) and (\ref{relation}), we can gain that,
\begin{equation*}
\begin{cases}
M\ {\rm unfold}({\mathcal X}) - N\ \rm{unfold}(\mathcal I_{11p})&=-\Lambda_1\ {\rm unfold}(\mathcal X)\\
N^\top\ {\rm unfold}({\mathcal X}) - P\ \rm{unfold}(\mathcal I_{11p}) &=-\Lambda_2\ \rm{unfold}(-\mathcal I_{11p})
\end{cases}
\end{equation*}
Writing the above formula into block forms,
\begin{equation*}
\left(
\begin{matrix}
M & N\\
N^\top & P
\end{matrix}
\right)
\left(
\begin{matrix}
\rm{unfold}(\mathcal X)\\
-\rm{unfold}(\mathcal I_{11p})
\end{matrix}
\right)
=
-\left(
\begin{matrix}
\Lambda_1 & 0\\
0 & \Lambda_2
\end{matrix}
\right)
\left(
\begin{matrix}
\rm{unfold}(\mathcal X)\\
-\rm{unfold}(\mathcal I_{11p})
\end{matrix}
\right).
\end{equation*}
Then we finish the proof of this theorem.
\end{proof}
\newpage

The estimate of the
normalized eigenvector of $\Gamma(\mathcal X)$ is given by
\begin{equation*}
\label{formula_z}
z^{(k)}=\frac{1}{1+\|\mathcal X\|_F^2}\left(
\begin{matrix}
\rm{unfold}(\mathcal X)\\
-\rm{unfold}(\mathcal I_{11p})
\end{matrix}
\right)
\end{equation*}
and we define the residual at step $k$ by
\begin{equation}
\rho ^{(k)}=\left\|\Gamma(\mathcal X)^{(k)}z^{(k)}+\Lambda(\mathcal X^{(k)}) z^{(k)}\right\|_F^2.
\end{equation}

\begin{breakablealgorithm}
\caption{Iterative Algorithm for the TR-TLS in Matrix Form}
	\label{Matrix Algorithm}
	\begin{algorithmic}[0] 
	\Require ~ $\mathcal A\in\mathbb R^{m\times n\times p}$, $\mathcal B \in\mathbb R^{n\times 1\times p}$, initial value of $\mathcal X^{(0)}$, error bound $\delta_0$, maximum of iteration steps $K_{\max}$
	
	\Ensure ~ Final value of $\mathcal X$\\
	\textbf{Step 1.} Calculate $\widetilde{\lambda_I}^{(0)}
	 \widetilde{\lambda_K}^{(0)}$ and $z^{(0)}$\\
	\textbf{Step 2.} Iterate to convergence \\
	\textbf{For} ${k}=0,1,2, \ldots, K_{\max}$
	
	(a) Solve $\left(\Gamma\left(\mathcal X^{(k)}\right)+\Lambda^{(k)}\right) y=z^{(k)}$
	
	(b) $\rm{unfold}(\mathcal X)^{(k+1)}=-y(1: np) / y(np+1)$
	
	(c) Update $\widetilde{\lambda_I}^{(k+1)}$ and $\widetilde{\lambda_K}^{(k+1)}$
	
	(d) $z^{(k+1)}=y /\|y\|_F$
	
	(e) Update $\rho^{(k+1)}$
	
	(f) \textbf{If} $$\frac{\|\mathcal X^{(k+1)} - \mathcal X^{(k)}\|_F}{\|\mathcal X^{(k)}\|_F}\leq \delta_0\ \ or\ \ \rho^{(k+1)}\leq\delta_0$$
		
	\ \ \ \ 	\textbf{break}\\
	\textbf{End For.}
	\end{algorithmic}
\end{breakablealgorithm}
\vspace{1em}

\begin{remark}
This Algorithm shares the similar form with Algorithm 1 in \cite{guo2002regularized}.
\end{remark} 
	
\subsection{Case II: Multi Lateral Slices}

The situation in multi lateral slices are often discussed in both color image and video deblurring. We repeat the calculation in the single lateral slices to achieve our goals. In next section, the numerical experiments of both color image and video deblurring are designed to illuminate the benefits of TR-TLS. Denote the map of computation in single lateral slices as 
\begin{equation}
\Phi(\mathcal A,\textbf B_i,\delta,K_{\max}) = \textbf X_i,
\end{equation}
where $\textbf{X}_i$ is the $i$-th lateral slice of tensor $\mathcal X$, i.e., $\textbf X_i=\mathcal X(:, i,:)$. Then the algorithm in multi lateral slices could be summarized as follows.
\vspace{1em}
\begin{breakablealgorithm}
\caption{Iterative Algorithm for the TR-TLS in multi lateral slices}
	\label{Multi}
	\begin{algorithmic}[0] 
	\Require ~ $\mathcal A\in\mathbb R^{m\times n\times p}$, $\mathcal B \in\mathbb R^{n\times s\times p}$, initial value of $\mathcal X^{(0)}$,  error bound $\delta_0$, maximum of iteration steps $K_{\max}$
	
	\Ensure ~ Final value of $\mathcal X$\\
	\textbf{For} $i=0,1,2,\ldots,s$
	
	Compute 
	$$
	\textbf X_i = \Phi(\mathcal A,\textbf B_i,\delta_0,K_{\max})
	$$	
	\textbf{End For.}
	\end{algorithmic}
\end{breakablealgorithm}
\vspace{0.5em}

\section{Applications and Numerical Examples}
In this section,we will test the algorithms mentioned above to demonstrate the effectiveness of TR-TLS method in applications to image and video deblurring. All experiments are carried out based on MATLAB-2021a on MacOS. The processor is Intel Core i5 with 8 GB RAM.

In this section, the original data is a tensor $\mathcal X\in\mathbb{R}^{n\times s\times p}$. Then after the map of the tensor $\mathcal A\in\mathbb{R}^{m\times n\times p}$, the observation is denoted as $\mathcal B\in\mathbb{R}^{m\times s\times p}$, i.e.,
\begin{equation}
\mathcal A_{\rm true} * \mathcal X_{\rm true} = \mathcal B_{\rm blur}.
\end{equation}

However, in practice, the observations of tensors $\mathcal A$ and $\mathcal B$ are usually subject to errors. Let us construct the error for two tensors $\mathcal A$ and $\mathcal B$, respectively. We first give the error tensor $\mathcal E$ with each element satisfying the standard normal distribution. Then the error tensors of $\mathcal A$ and $\mathcal B$ can be defined as follows,
\begin{equation*}
{\mathcal{E}}_{\mathcal A,j}:=\eta \frac{{\mathcal{E}}_{ j}}{\left\|{\mathcal{E}}_{ j}\right\|_F}\left\|{\mathcal{A}}_{\text {true }, j}\right\|_F, \quad j=1,2, \ldots, p,
\end{equation*}
and
\begin{equation*}
{\mathcal{E}}_{\mathcal B,j}:=\eta \frac{{\mathcal{E}}_{ j}}{\left\|{\mathcal{E}}_{ j}\right\|_F}\left\|{\mathcal{B}}_{\text {true }, j}\right\|_F, \quad j=1,2, \ldots, p,
\end{equation*}
where $\eta$ is set as $\eta=0.001$ in this section. 

What we should solve is a tensor regularized total least squares problem,
\begin{equation}
\begin{aligned}
\min_{\tilde{\mathcal{A}}, \tilde{\mathcal{B}}, \mathcal X} \left\|(\mathcal{A}, \mathcal{B})-\left(\tilde{\mathcal{A}}, \tilde{\mathcal{A}} * \mathcal{X}\right)\right\|_F \ \ \ \text { s.t. } \tilde{\mathcal{B}}=\tilde{\mathcal{A}} * \mathcal{X}, \quad\left\|\mathcal{K} * \mathcal{X}\right\|_F \leq \delta,
\end{aligned}
\end{equation}
of the given equations,
\begin{equation}
(\mathcal A_{\rm true} +\mathcal E_{\mathcal A}) * \mathcal X \approx \mathcal B_{\rm true} + \mathcal E_{\mathcal B}.
\end{equation}
Here, the tensor $\mathcal A_{\rm true}$ is generated \cite{doicu2010numerical} by
\begin{equation*}
\begin{cases}
\mathrm{z}=\left[\exp \left(-\left([0: \operatorname{band}-1] .^2\right) /\left(2 \sigma^2\right)\right),  \operatorname{zeros}(1, \mathrm{~N}-\mathrm{band})\right], \\
A=\frac{1}{\sigma \sqrt{2 \pi}} \text { toeplitz }([\mathrm{z}(1) \mathrm{fliplr}(\mathrm{z}(2: \text {end }))], \mathrm{z}),\\ \mathcal{A}_{true}^{(i)}=A(i, 1) A, \quad i=1,2, \ldots, 256,
\end{cases}
\end{equation*}
where $ N= 256,
\sigma = 4$, $band = 7$ and $A(i,1)$ means the first column in the matrix $A$. Following these operations, the condition number for $i$-th tensor frontal slice $\mathcal A^{(i)}$ satisfies that 
$
\text{cond}\left(\mathcal A^{(i)}\right)\approx5.35 \times 10^9,
$
for $i=1,2,\ldots,12$. The other slices have infinite condition number, i.e.,
$
\text{cond}\left(\mathcal A^{(i)}\right)=+ \infty,
$
for $ i>12$. Moreover, after the FFT operation, the condition number of each tensor frontal slice satisfies  $\text{cond}\left(\hat{\mathcal A}^{(i)}\right)\approx5.35 \times 10^9$, for $i=1,2,\ldots,256$, where $\hat{\mathcal A}=\text{fft}\left(\mathcal A,\ [], \ 3\right)$.

Meanwhile, we will use two regularization operators $\mathcal{K}_1 \in \mathbb{R}^{(m-2) \times m \times n}$ and $\mathcal{K}_2 \in \mathbb{R}^{(m-1) \times m \times n}$. The tensor $\mathcal{K}_1$ has a tridiagonal matrix as its first frontal slice,
$$
\mathcal{K}_1^{(1)}=\frac{1}{4}\left[\begin{array}{cccccc}
-1 & 2 & -1 & && \\
&-1&2&-1&&\\
&& \ddots & \ddots & \ddots & \\
&& & -1 & 2& -1
\end{array}\right] \in \mathbb{R}^{(m-2) \times m}
$$
and the remaining frontal slices $\mathcal{K}_1^{(i)} \in \mathbb{R}^{(m-2) \times m}, (i=2,3, \ldots, n)$, are zero matrices. The first frontal slice of the regularization operator $\mathcal{K}_2 \in \mathbb{R}^{(m-1) \times m \times n}$ is the bidiagonal matrix
$$
\mathcal{K}_2^{(1)}=\frac{1}{2}\left[\begin{array}{ccccc}
1 & -1 & & & \\
& 1 & -1 & & \\
& & \ddots & \ddots & \\
& & & 1 & -1
\end{array}\right] \in \mathbb{R}^{(m-1) \times m}
$$
and the remaining frontal slices $\mathcal{K}_2^{(i)} \in \mathbb{R}^{(m-1) \times m}$, ($i=2,3, \ldots, n$) are zero matrices. Other regularization operators of interest can be defined similarly.
\subsection{Applications in Image and Video Deblurring}
In this section, we mainly show the intuitive image deblurring results of the experiment. The data of grayscale images are stored by single lateral slice tensor. Color images and video data are stored by multi lateral slices and higher order tensors. All data is normalized into the interval as $[0, 1]$.

\subsubsection{Grayscale Images}
We utilize the TR-TLS method in the single lateral slices to model the deblurring problem of grayscale images. Single lateral slice tensor $\mathcal X\in \mathbb{R}^{m\times 1\times n}$ stores the data of image after ``twist'' operation, where ``twist'' is the inverse operation of ``squeeze''，
\begin{equation}
\mathcal X=\operatorname{twist}({X}) \Rightarrow {\mathcal{X}}(i, 1, j) = X(i, j).
\end{equation}

The operations, ``twist'' and ``squeeze'', build the bridge between single lateral slice tensor $\mathcal X\in \mathbb{R}^{m\times 1\times n}$ and the matrix $X\in\mathbb{R}^{m\times n}$ \cite{kilmer2013third}.
\begin{figure}[htbp]
\centering
\includegraphics[width=2in]{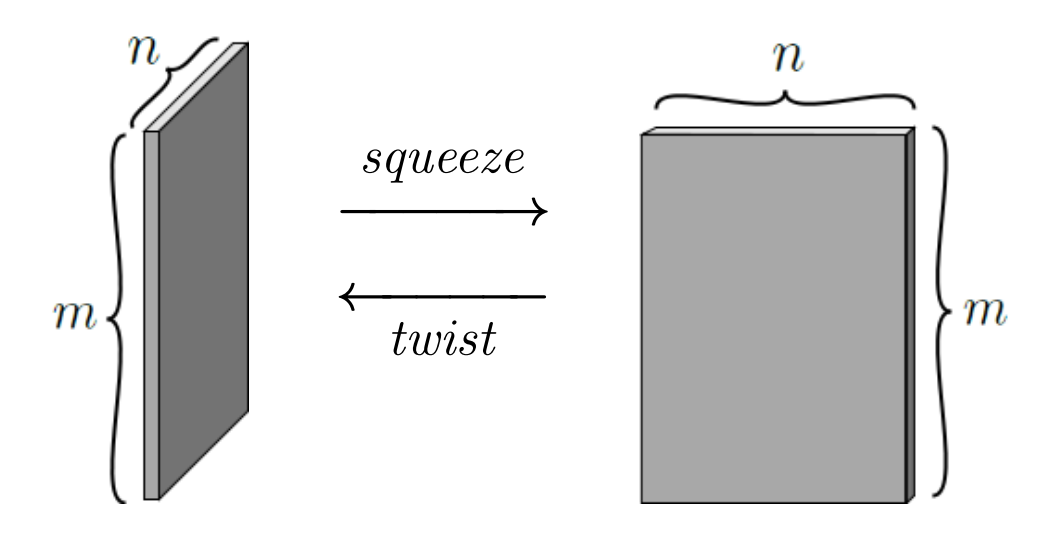}
\caption{$
m \times 1 \times n \text { tensors and } m \times n \text { matrices related through the squeeze and twist operations. }
$}
\end{figure}

There are two images tested in this case, using Algorithm \ref{Tensor Algorithm}. The following figures are listed to illuminate the effectiveness of the TR-TLS method. After the TR-TLS, we use ``squeeze'' operation to convert the tensor into a displayable image matrix.

\begin{figure}[htbp]
	\centering
	\subfigure[Original Image]{\includegraphics[width=1.3in]{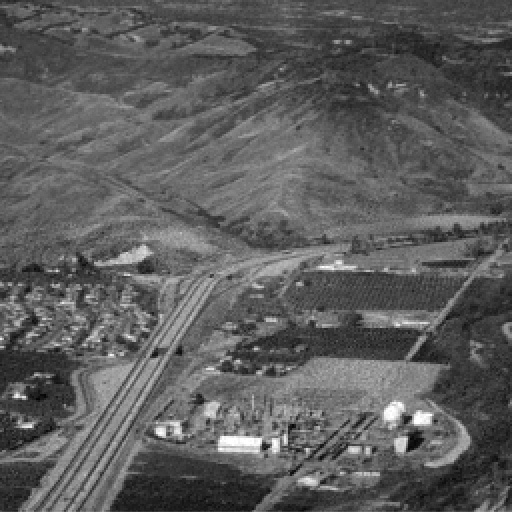}}
		\hspace{1cm}
	\subfigure[Blurred Imgae]{\includegraphics[width=1.3in]{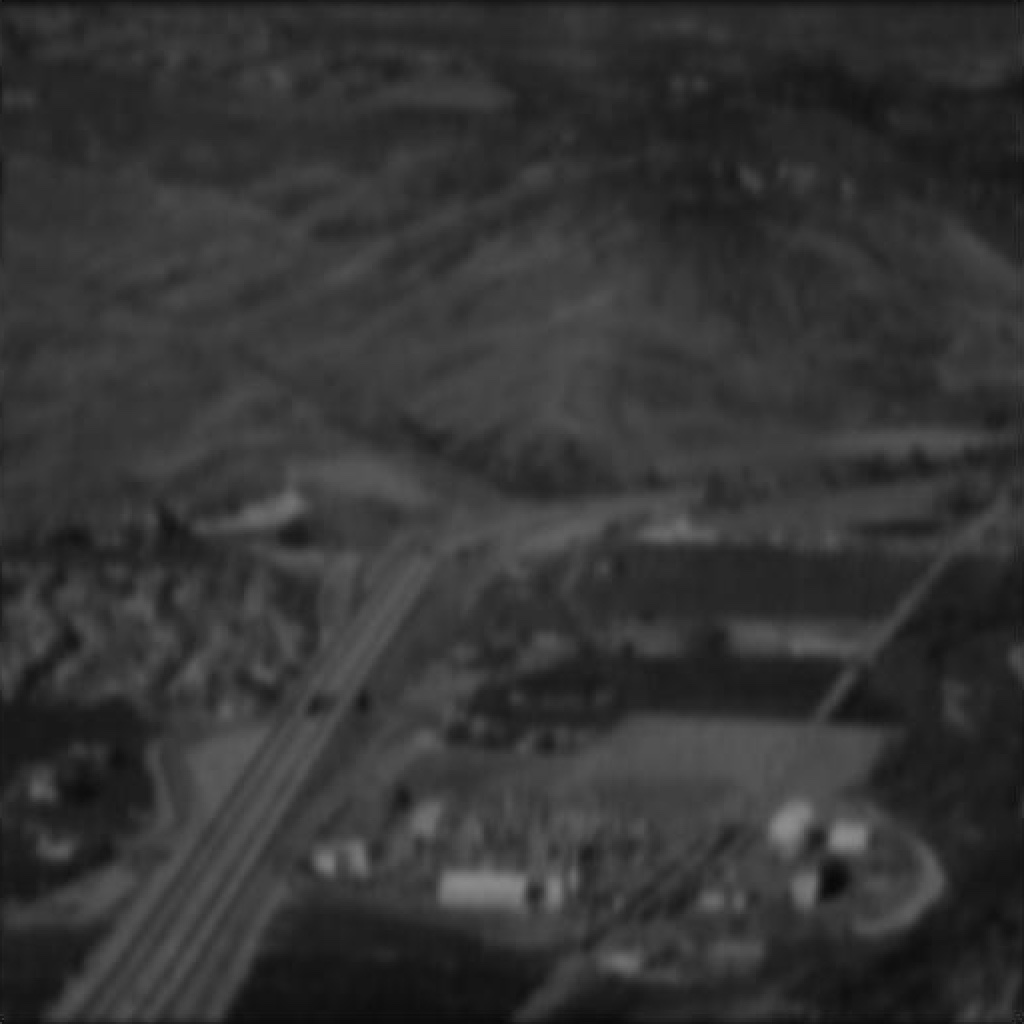}}
		\hspace{1cm}
	\subfigure[Debluurred Image]{\includegraphics[width=1.3in]{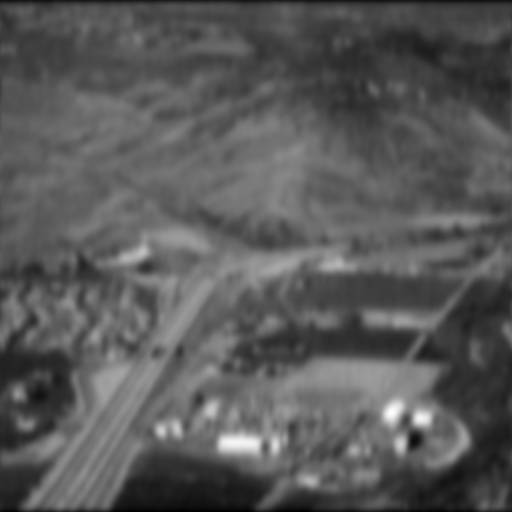}}
	\caption{Blurred and Deblurred results of City Image}
\end{figure}
\begin{figure}[htbp]
	\centering
	\subfigure[Original Image]{\includegraphics[width=1.3in]{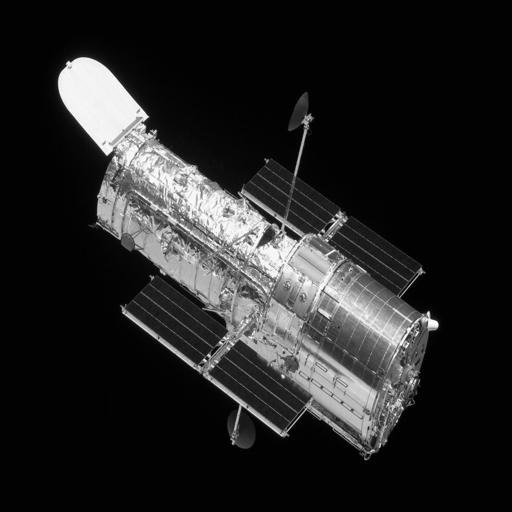}}
		\hspace{1cm}
	\subfigure[Blurred Imgae]{\includegraphics[width=1.3in]{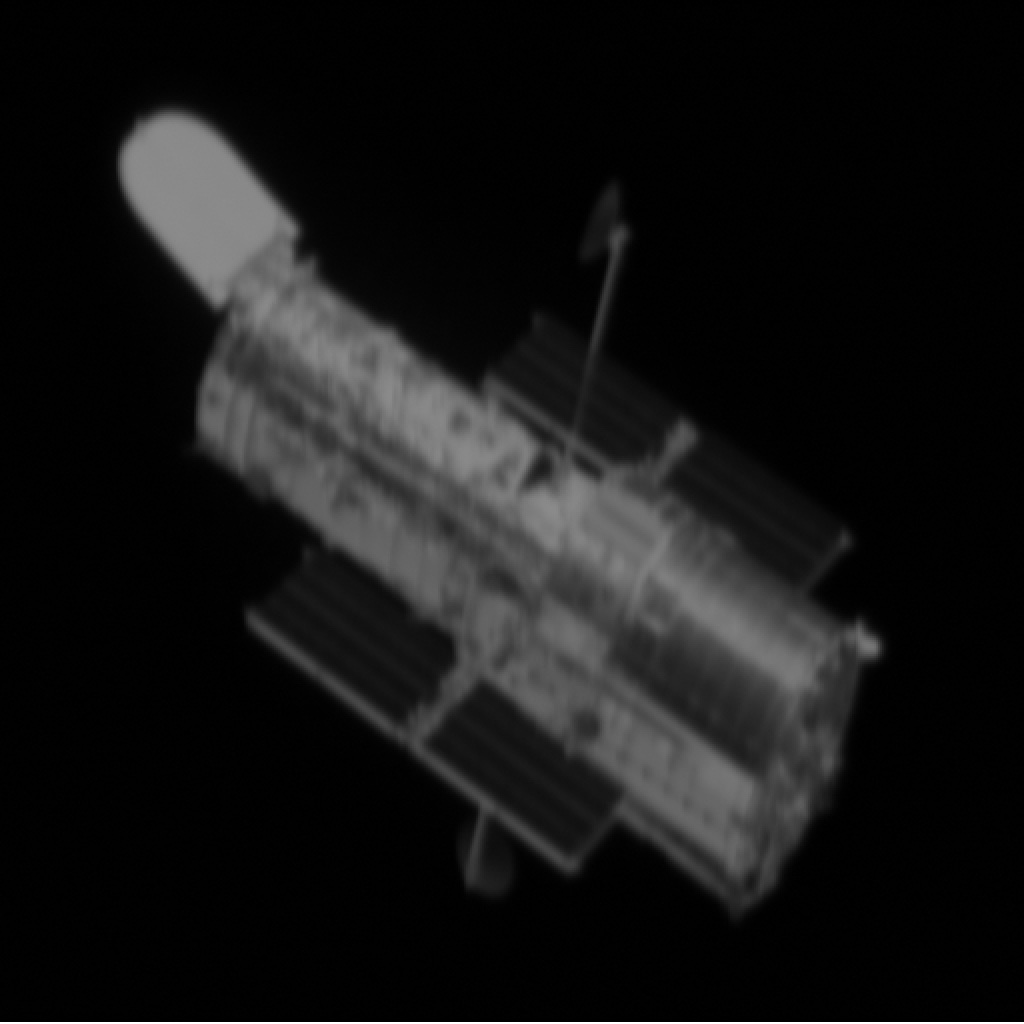}}
		\hspace{1cm}
	\subfigure[Debluurred Image]{\includegraphics[width=1.3in]{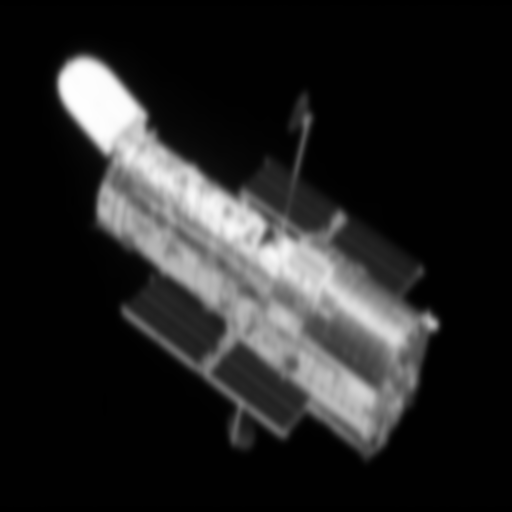}}
	\caption{Blurred and Deblurred results of Artificial Satellite Image}
\end{figure}

%

\subsubsection{Color Images}

For solving  deblurring problems, storage and modeling methods in color image are similar to those of grayscale image cases. The RGB image could be modeled as a third order tensor $\mathcal X\in\mathbb{R}^{m\times n\times 3}$, which need to be twisted into a tensor as $\hat{\mathcal X}\in\mathbb{R}^{m\times 3\times n}$. We only need to deblur three layers of RGB data respectively, and then concatenate them together to get the required restored image, using the TR-TLS in multi lateral slices with Algorithm 4.

\begin{figure}[htbp]
	\centering
	\subfigure[Original Image]{\includegraphics[width=1.3in]{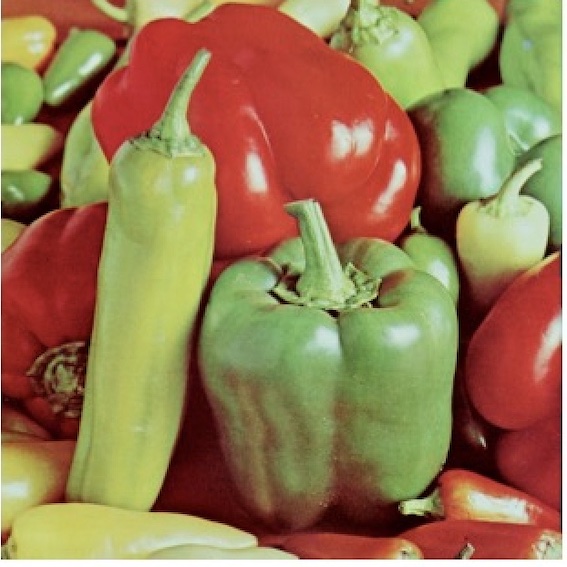}}
		\hspace{1cm}
	\subfigure[Blurred Imgae]{\includegraphics[width=1.3in]{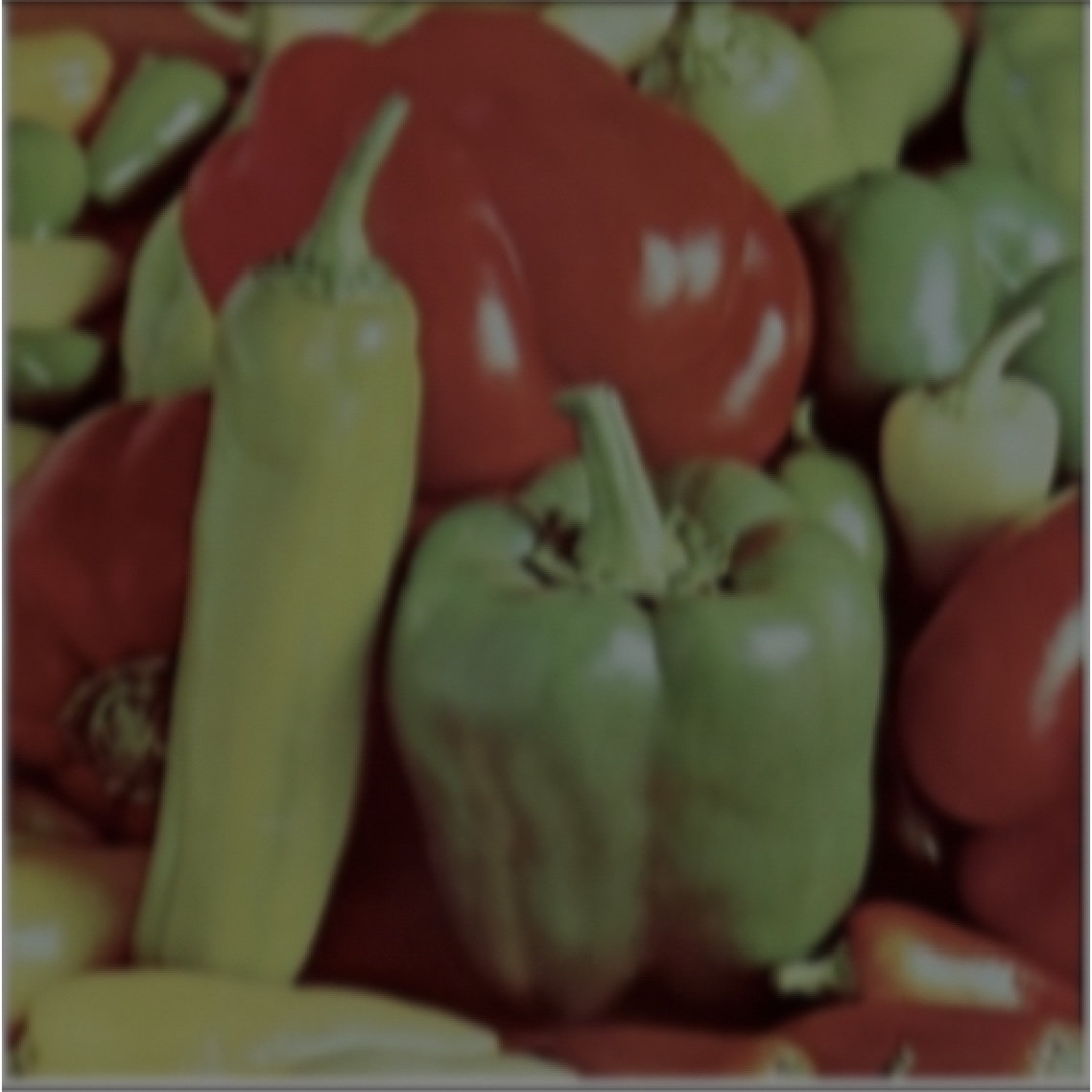}}
				\hspace{1cm}
	\subfigure[Debluurred Image]{\includegraphics[width=1.3in]{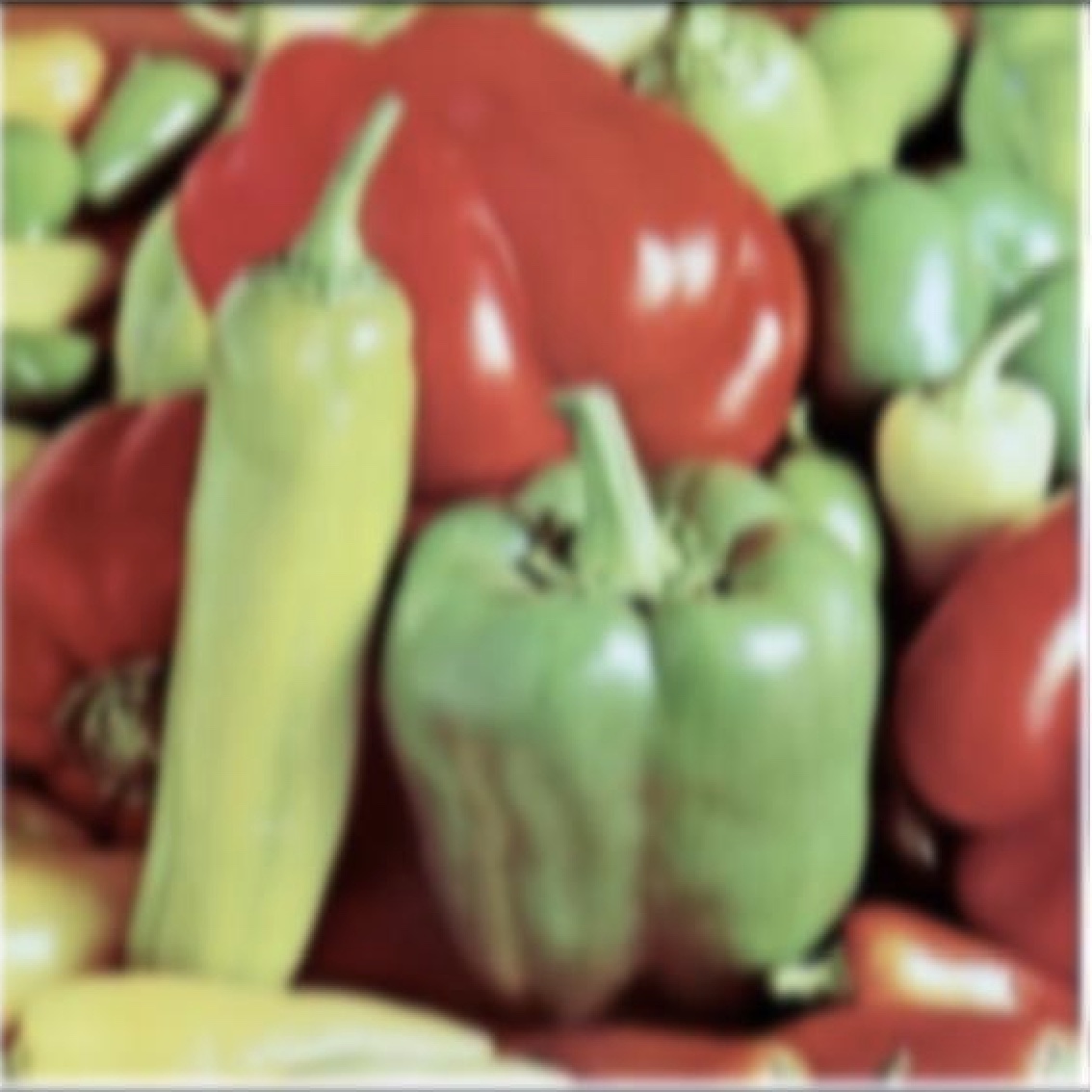}}
		\hspace{1cm}
	\subfigure[Debluurred Image in R-Dimension]{\includegraphics[width=1.3in]{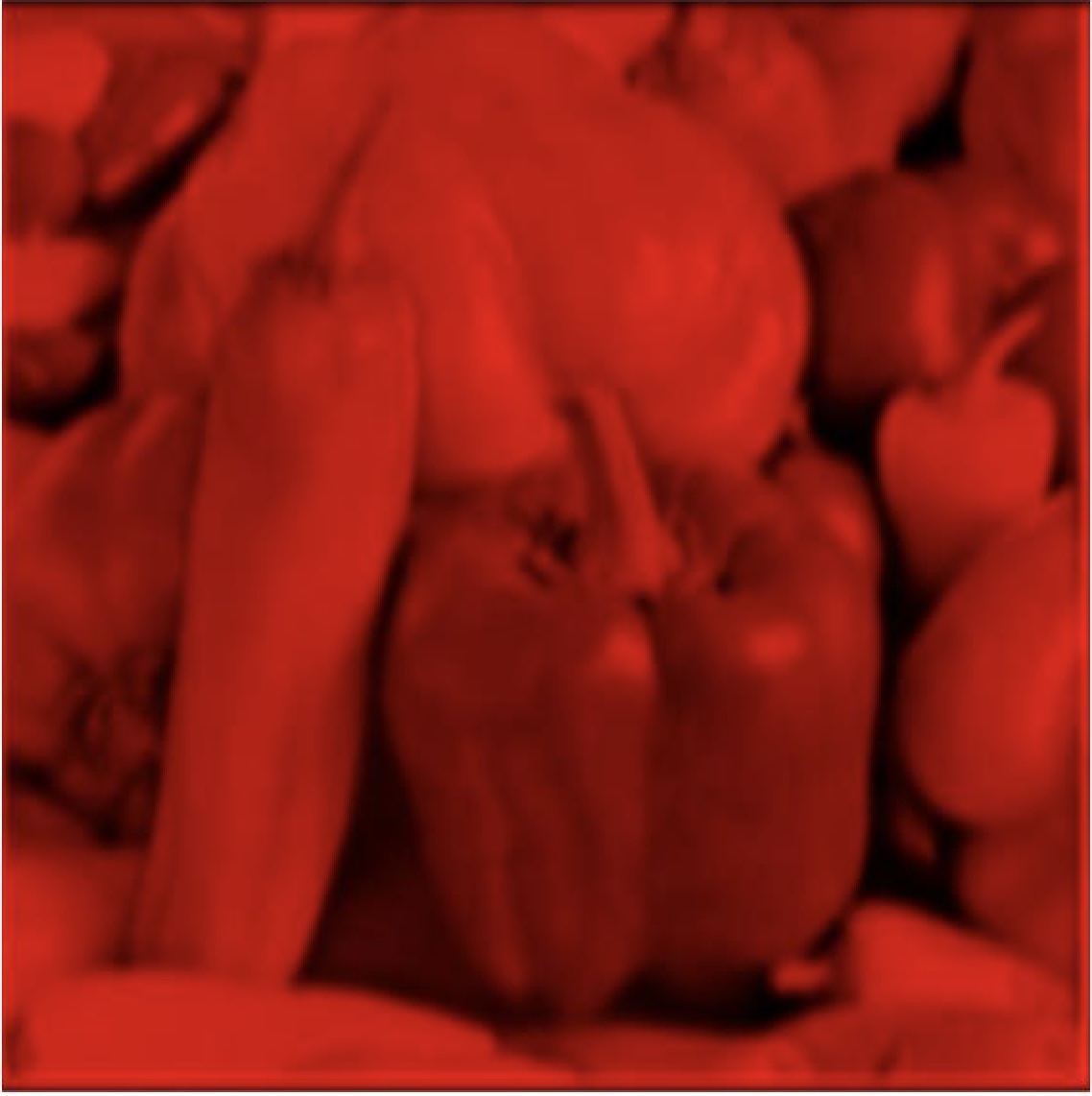}}
			\hspace{1cm}
	\subfigure[Debluurred Image  in G-Dimension]{\includegraphics[width=1.3in]{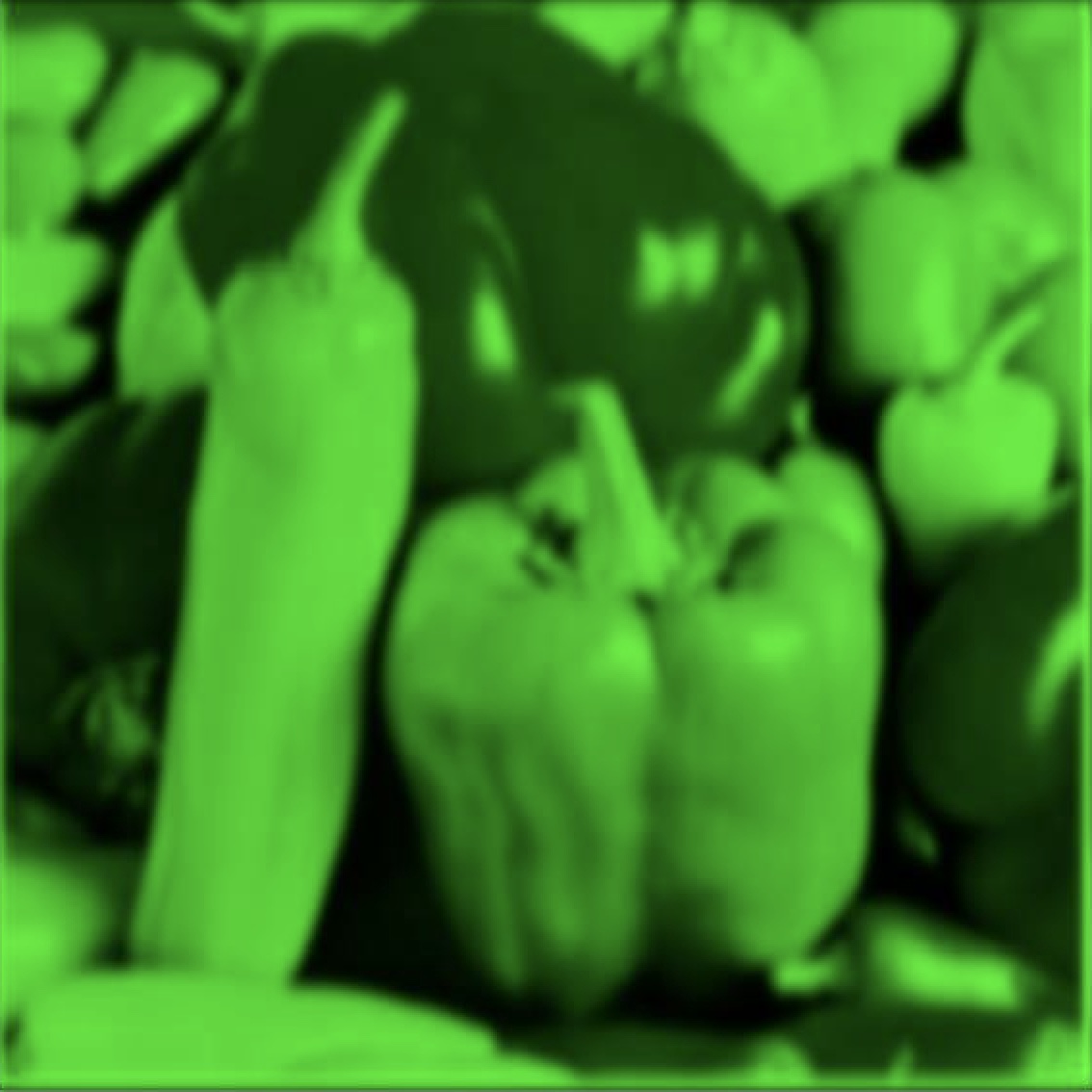}}
			\hspace{1cm}
		\subfigure[Debluurred Imgae  in B-Dimension]{\includegraphics[width=1.3in]{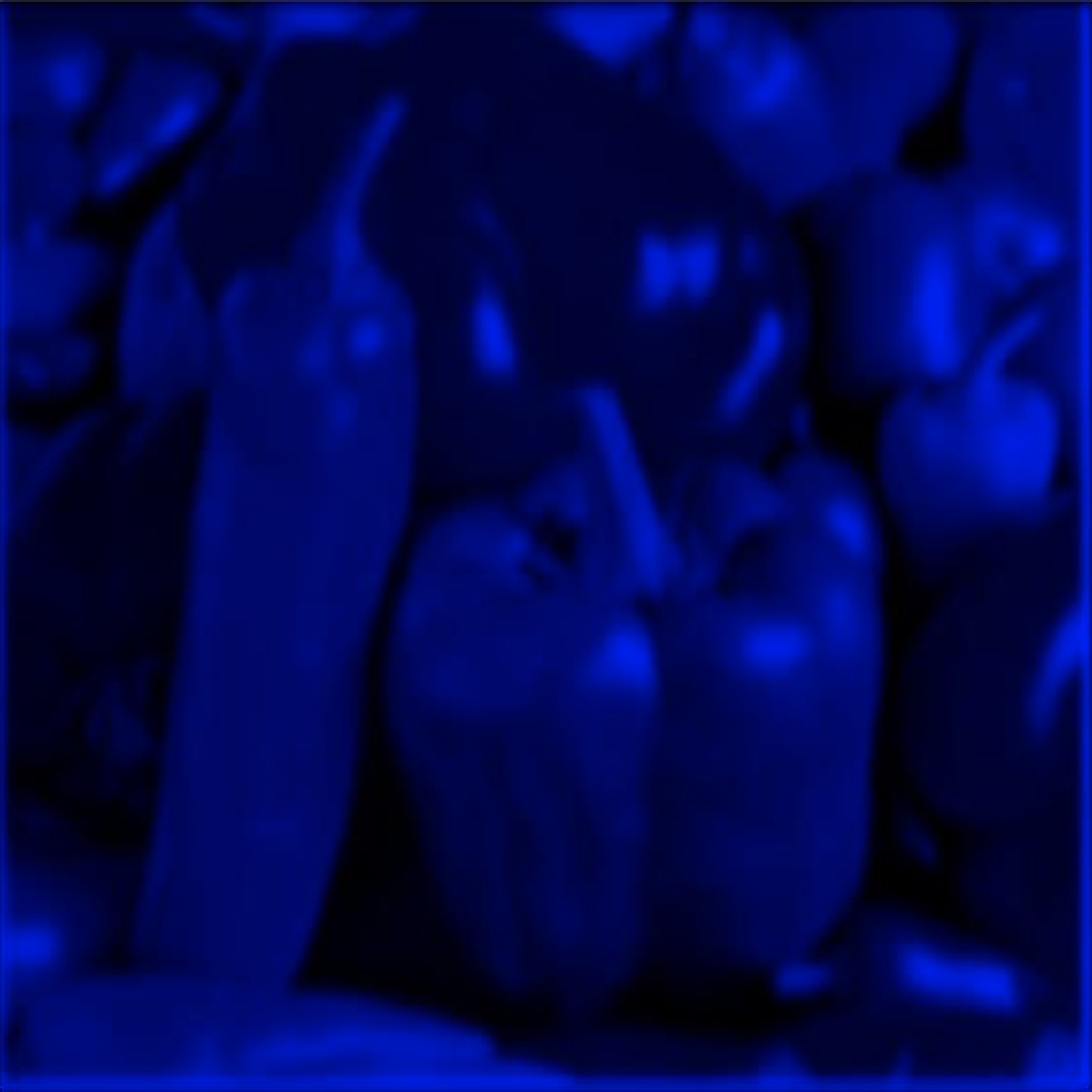}}

	\caption{Blurred and Deblurred results of Papper in RGB dimensions}
\end{figure}

\begin{figure}[htbp]

	\centering
	\subfigure[Original Image]{\includegraphics[width=1.3in]{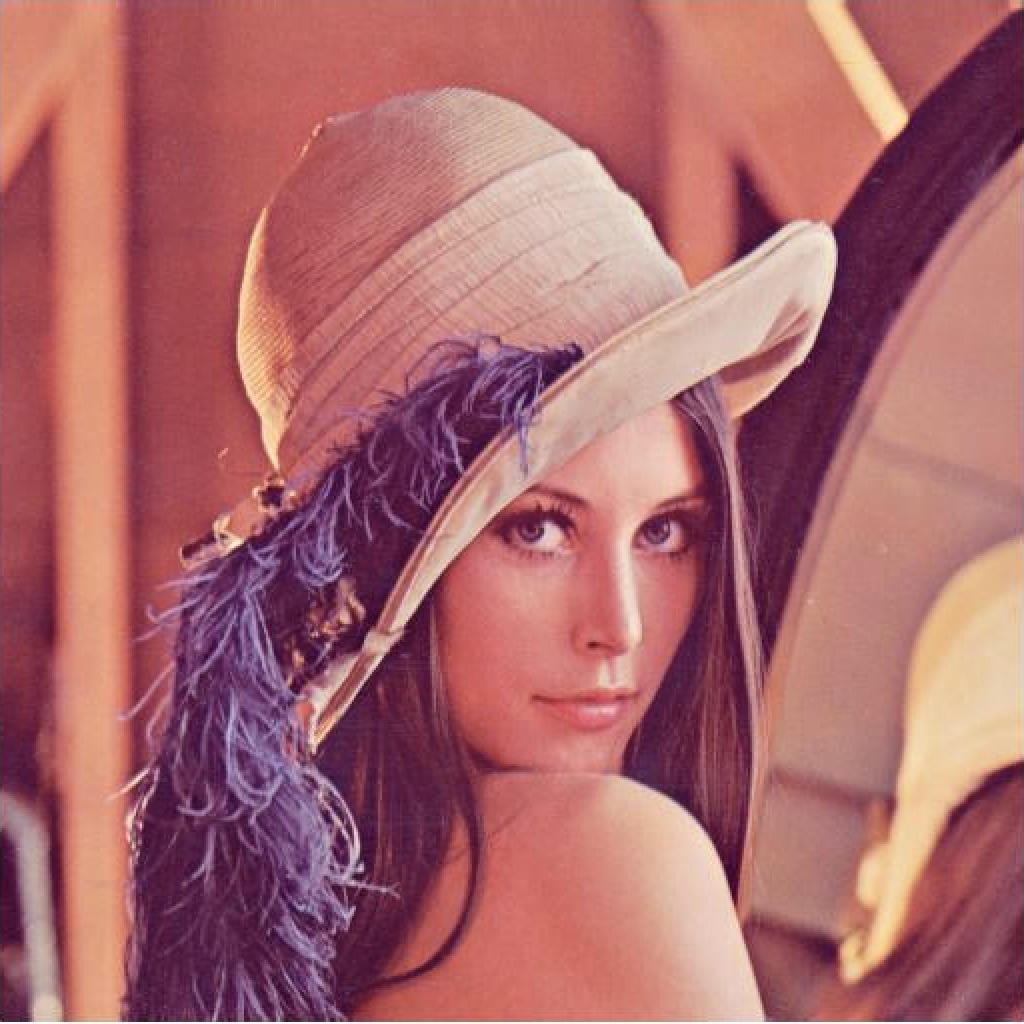}}
		\hspace{1cm}
	\subfigure[Blurred Imgae]{\includegraphics[width=1.3in]{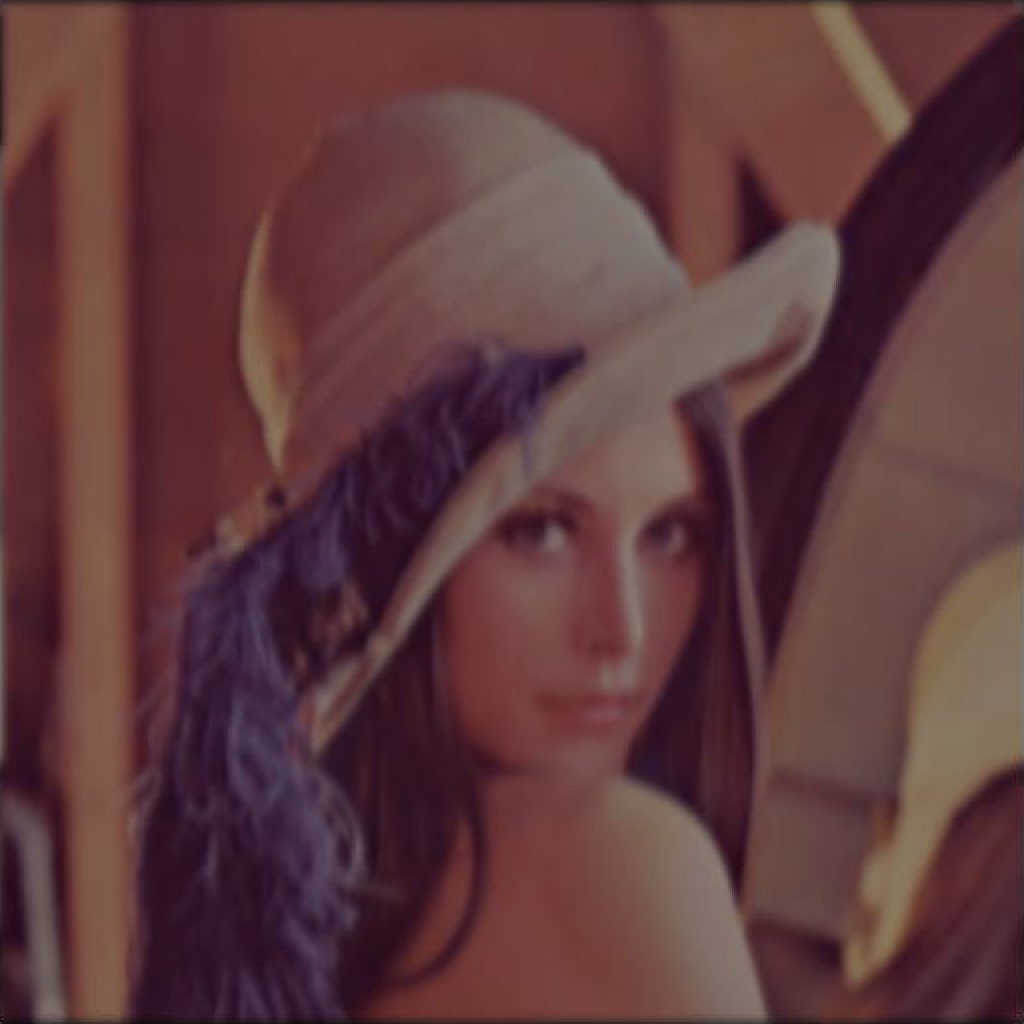}}
				\hspace{1cm}
	\subfigure[Debluurred Image]{\includegraphics[width=1.3in]{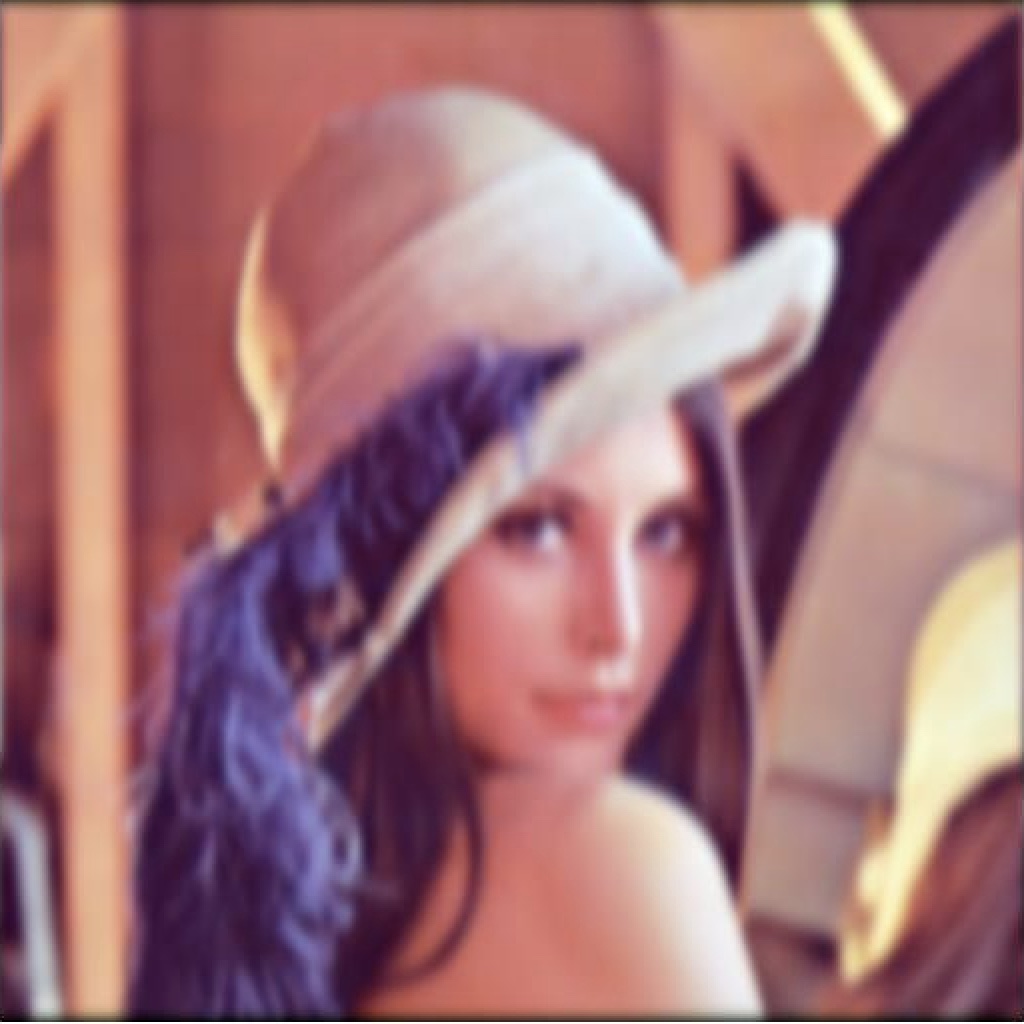}}
		\hspace{1cm}
	\subfigure[Debluurred Image in R-Dimension]{\includegraphics[width=1.3in]{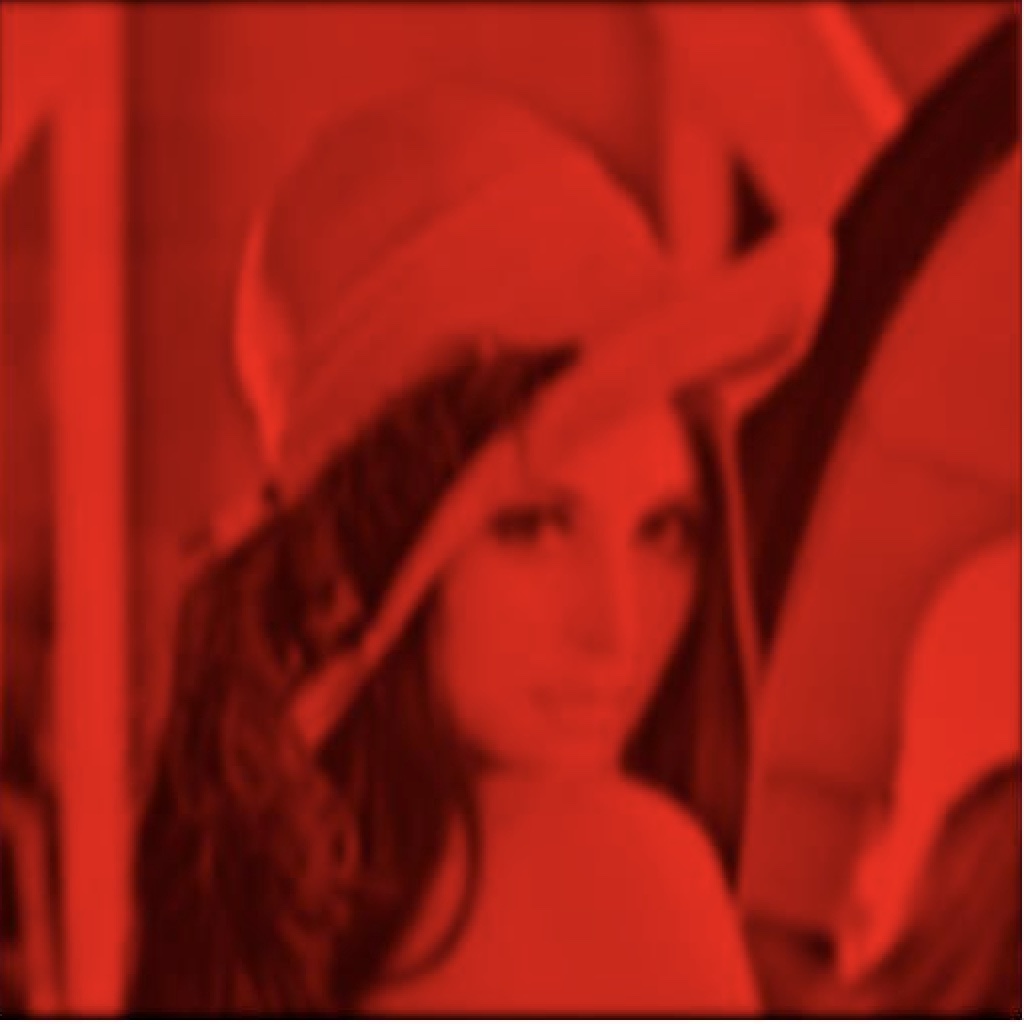}}
			\hspace{1cm}
	\subfigure[Debluurred Image  in G-Dimension]{\includegraphics[width=1.3in]{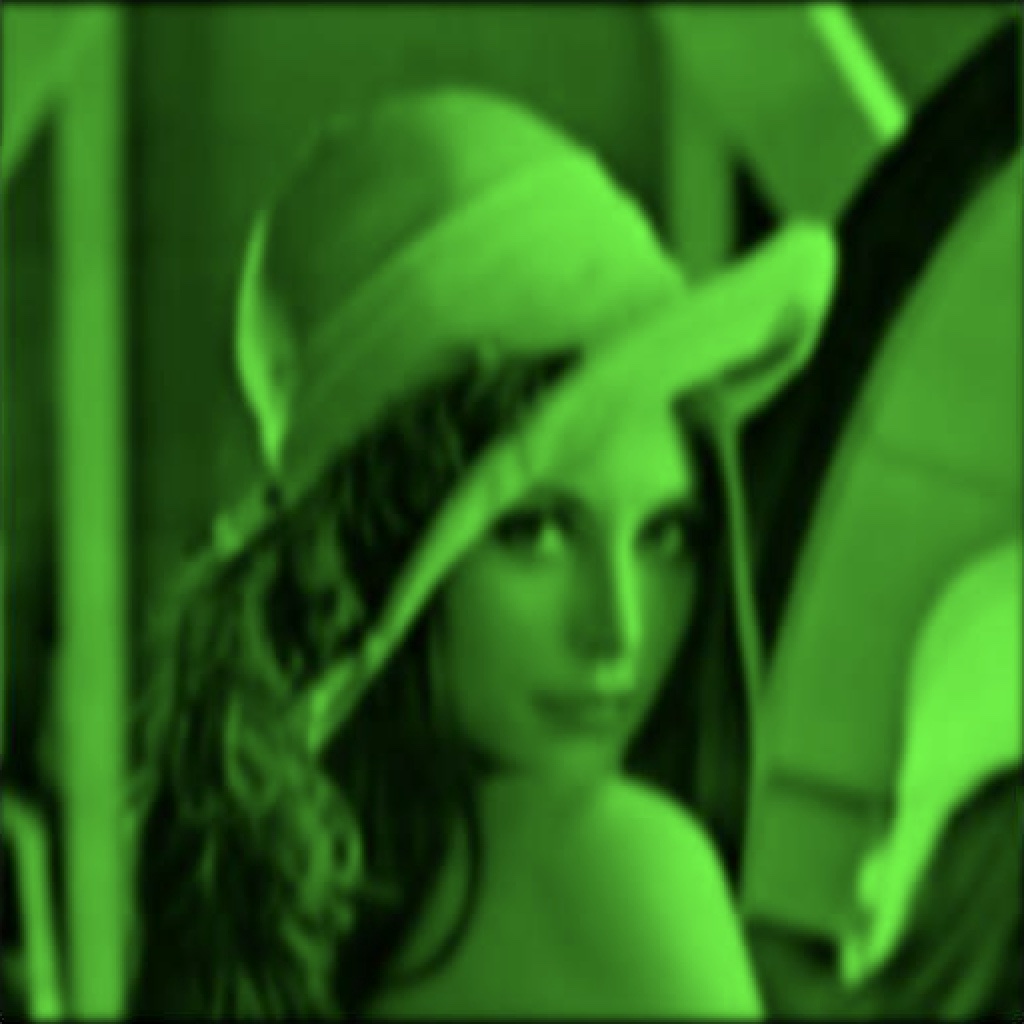}}
			\hspace{1cm}
		\subfigure[Debluurred Imgae  in B-Dimension]{\includegraphics[width=1.3in]{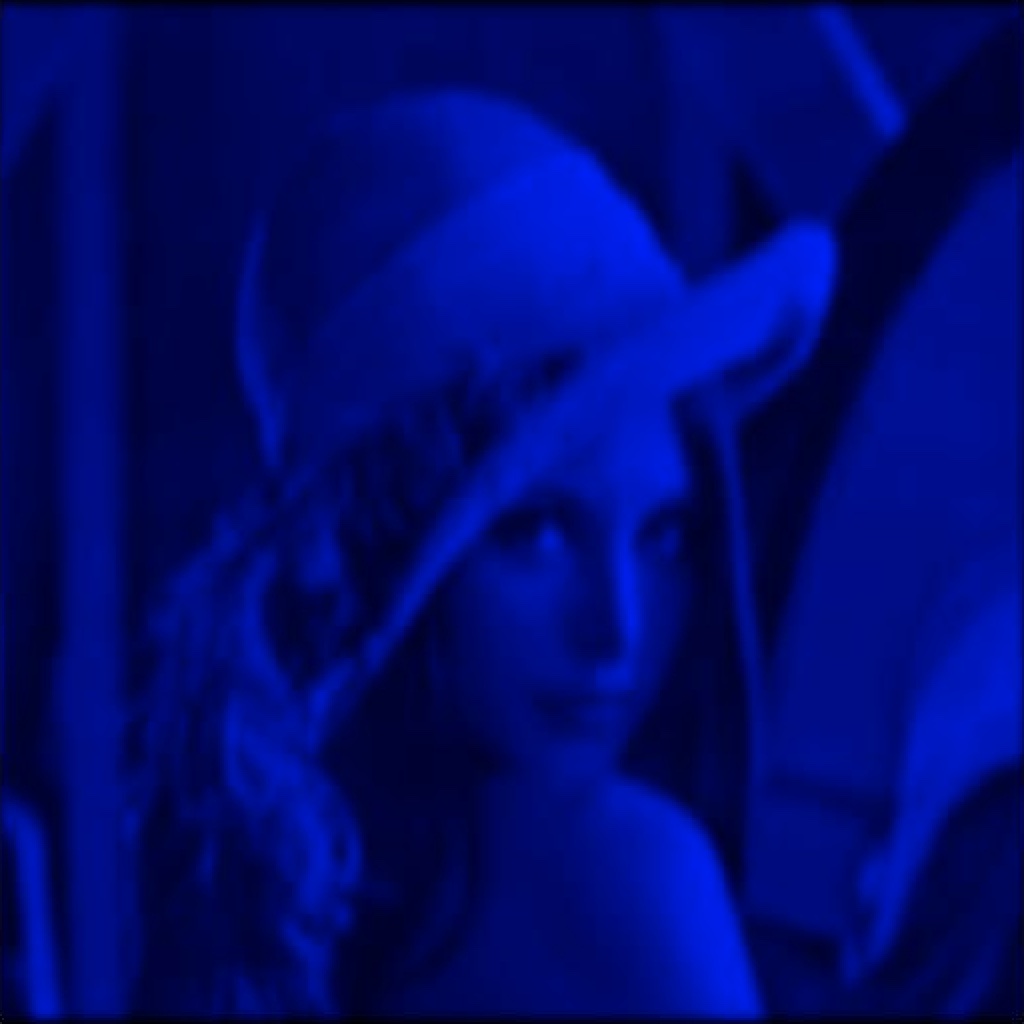}}

	\caption{Blurred and Deblurred results of Lena in RGB dimensions}
\end{figure}

\subsubsection{Videos}
Video data deblurring adopts the same strategy as color image deblurring. The only difference is the number of lateral layers on the model. Gray video data is stored in $\mathcal X_{\rm video}\in\mathbb{R}^{n\times s\times p}$, while image data is stored in $\mathcal X_{\rm image}\in\mathbb{R}^{n\times 1\times p}$ or $\mathcal X_{\rm image}\in\mathbb{R}^{n\times 3\times p}$, where $s$ could be any positive integer. Meanwhile, color video deblurring problems can be solved by the TR-TLS based on the T-product in higher-order cases \cite{martin2013order}, which repeats the TR-TLS  in single lateral slices, too.

\begin{figure}[htbp]
	\centering
	\subfigure{\includegraphics[width=0.71in]{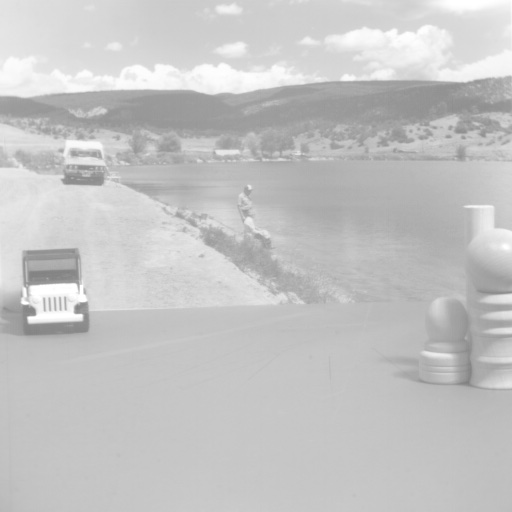}}
		\hspace{0.1cm}
	\subfigure{\includegraphics[width=0.71in]{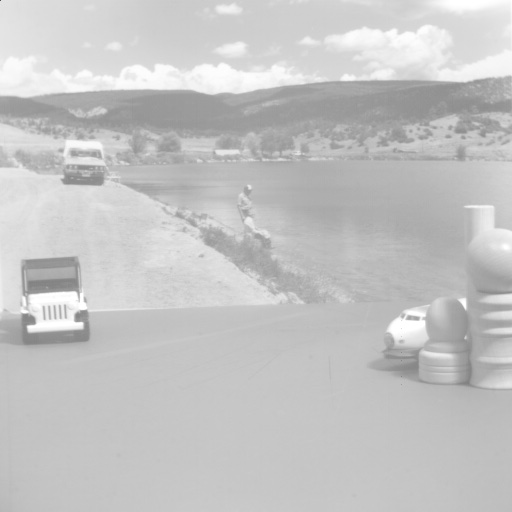}}
			\hspace{0.1cm}
	\subfigure{\includegraphics[width=0.71in]{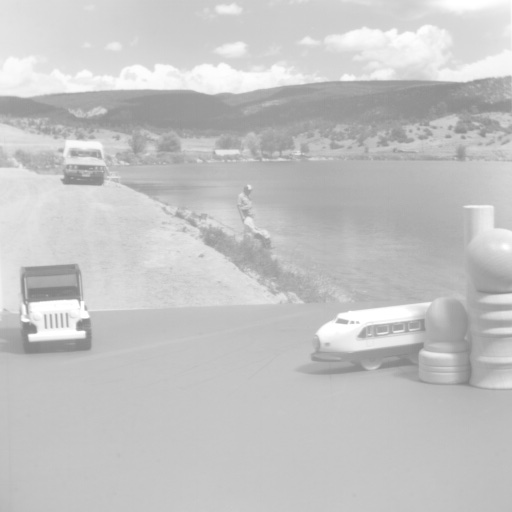}}
			\hspace{0.1cm}
	\subfigure{\includegraphics[width=0.71in]{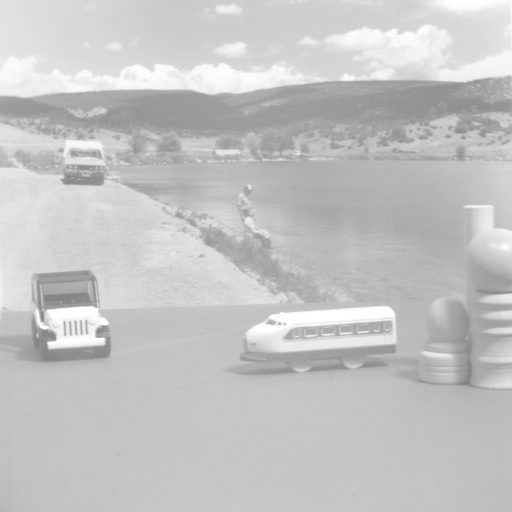}}
			\hspace{0.1cm}
	\subfigure{\includegraphics[width=0.71in]{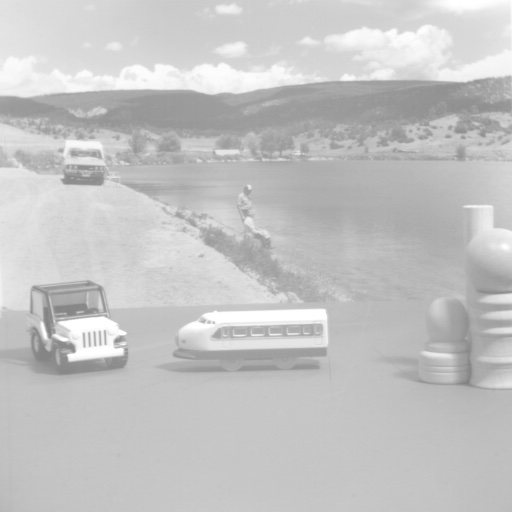}}
			\hspace{0.1cm}\\
			
	\subfigure{\includegraphics[width=0.71in]{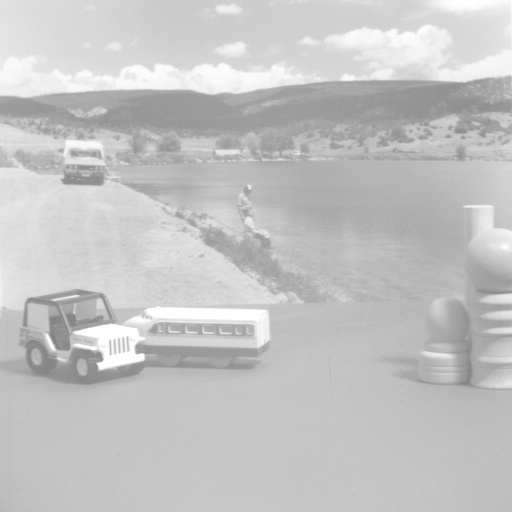}}
			\hspace{0.1cm}
	\subfigure{\includegraphics[width=0.71in]{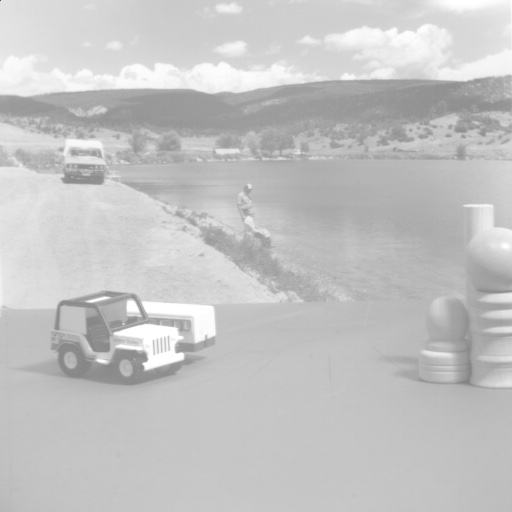}}
			\hspace{0.1cm}
	\subfigure{\includegraphics[width=0.71in]{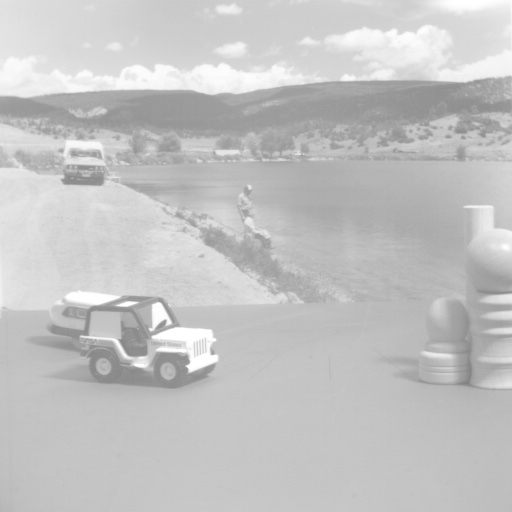}}
			\hspace{0.1cm}
	\subfigure{\includegraphics[width=0.71in]{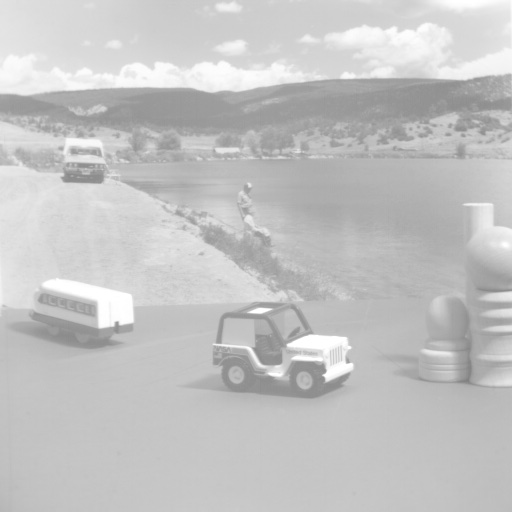}}
			\hspace{0.1cm}
	\subfigure{\includegraphics[width=0.71in]{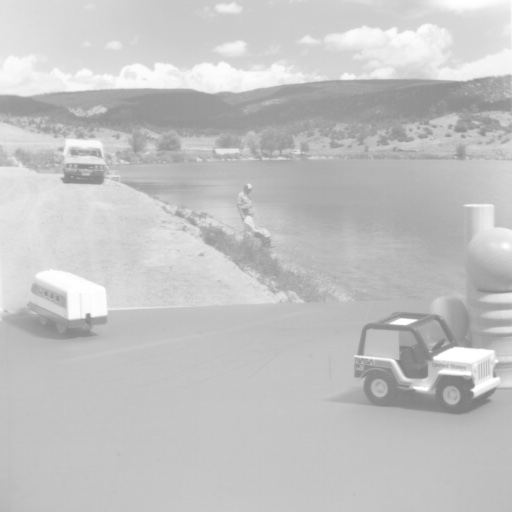}}
			\hspace{0.1cm}\\
\ 
	\subfigure{\includegraphics[width=0.71in]{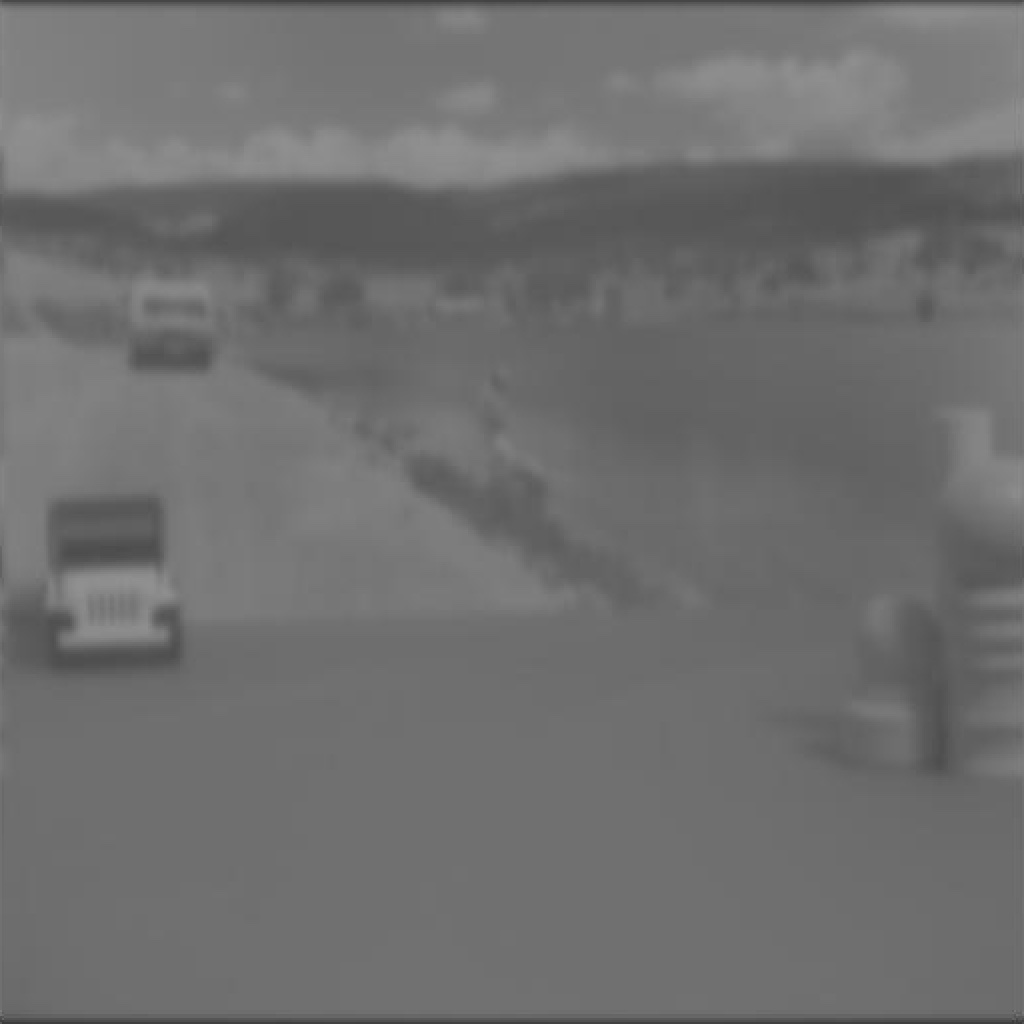}}
		\hspace{0.1cm}
	\subfigure{\includegraphics[width=0.71in]{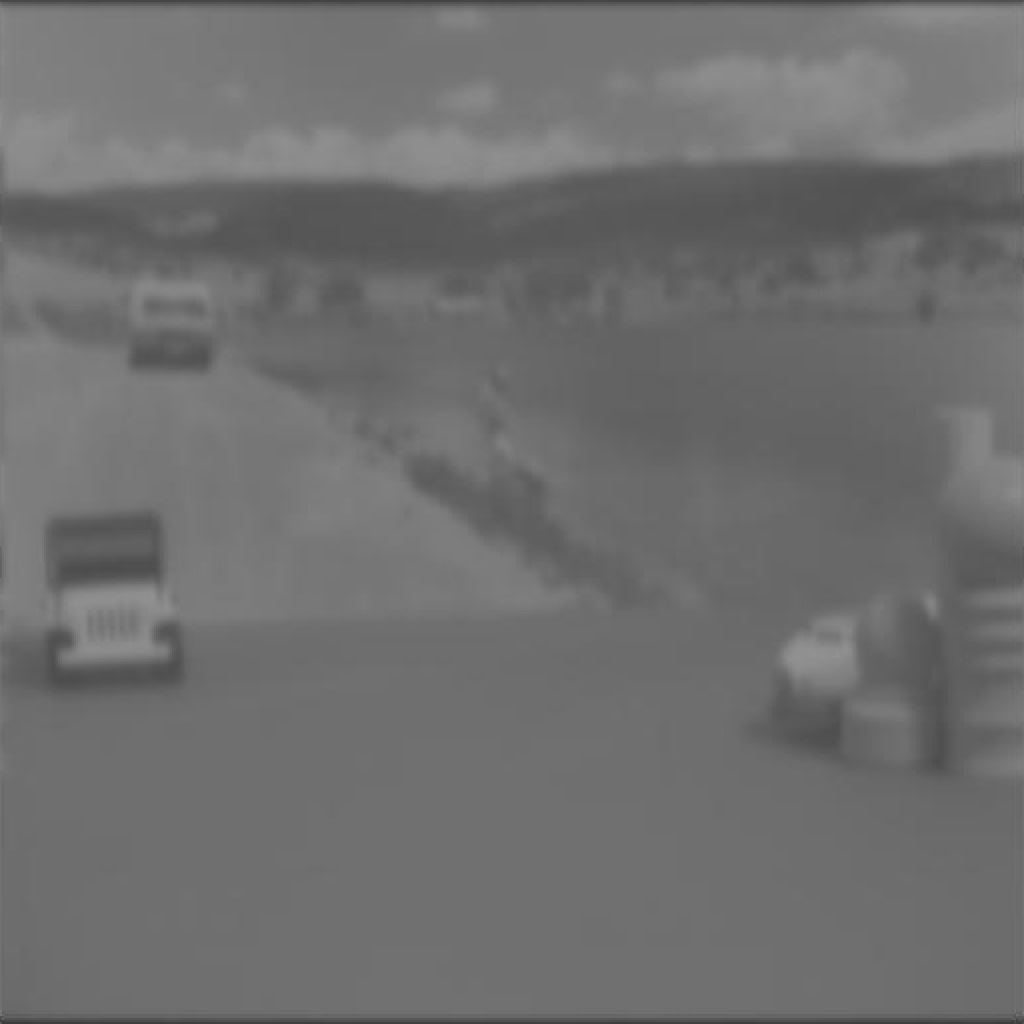}}
			\hspace{0.1cm}
	\subfigure{\includegraphics[width=0.71in]{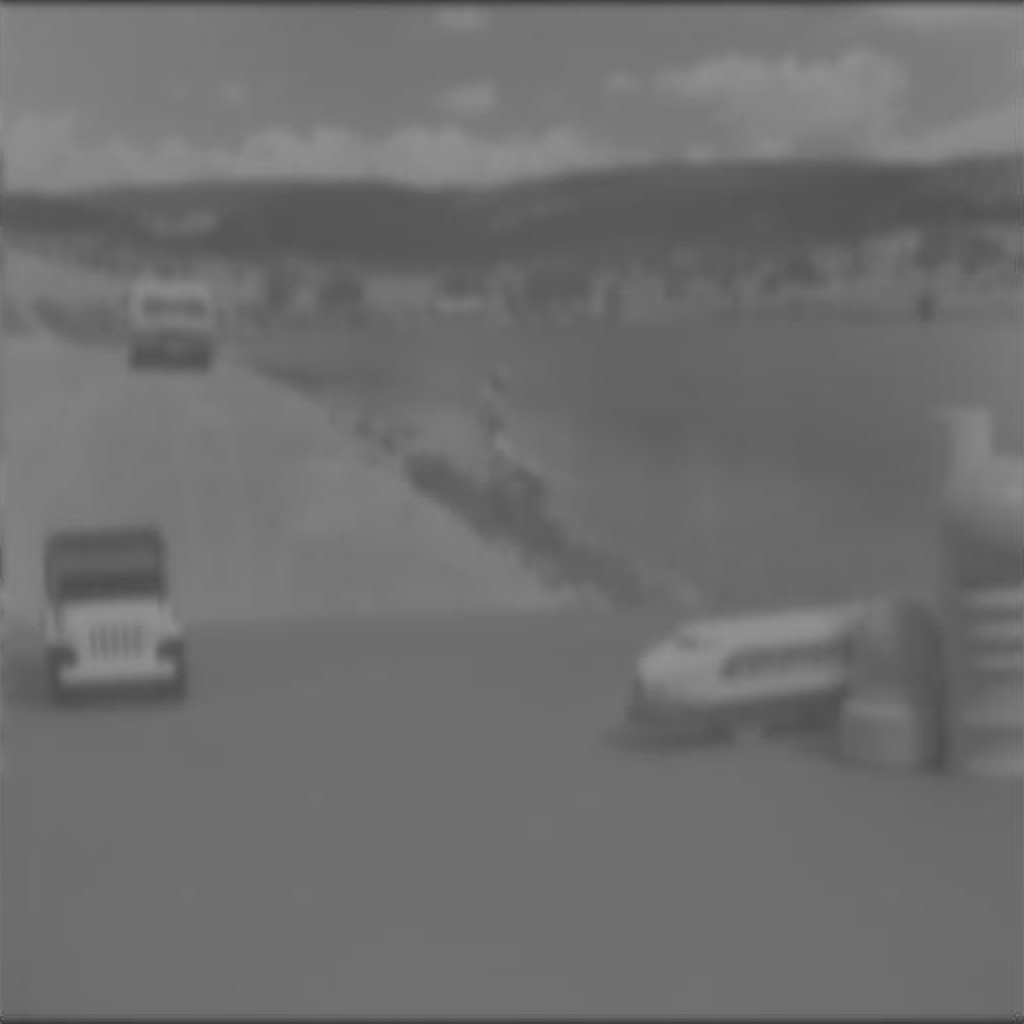}}
			\hspace{0.1cm}
	\subfigure{\includegraphics[width=0.71in]{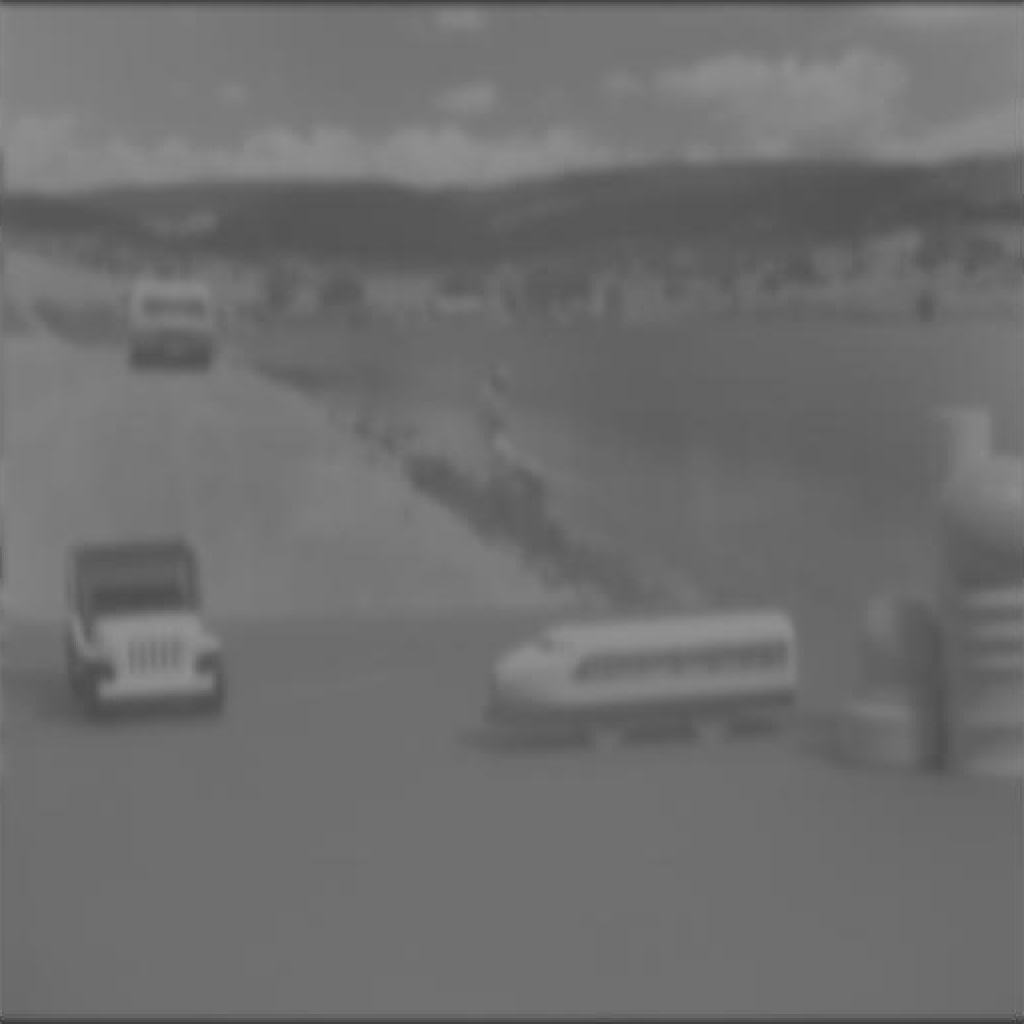}}
			\hspace{0.1cm}
	\subfigure{\includegraphics[width=0.71in]{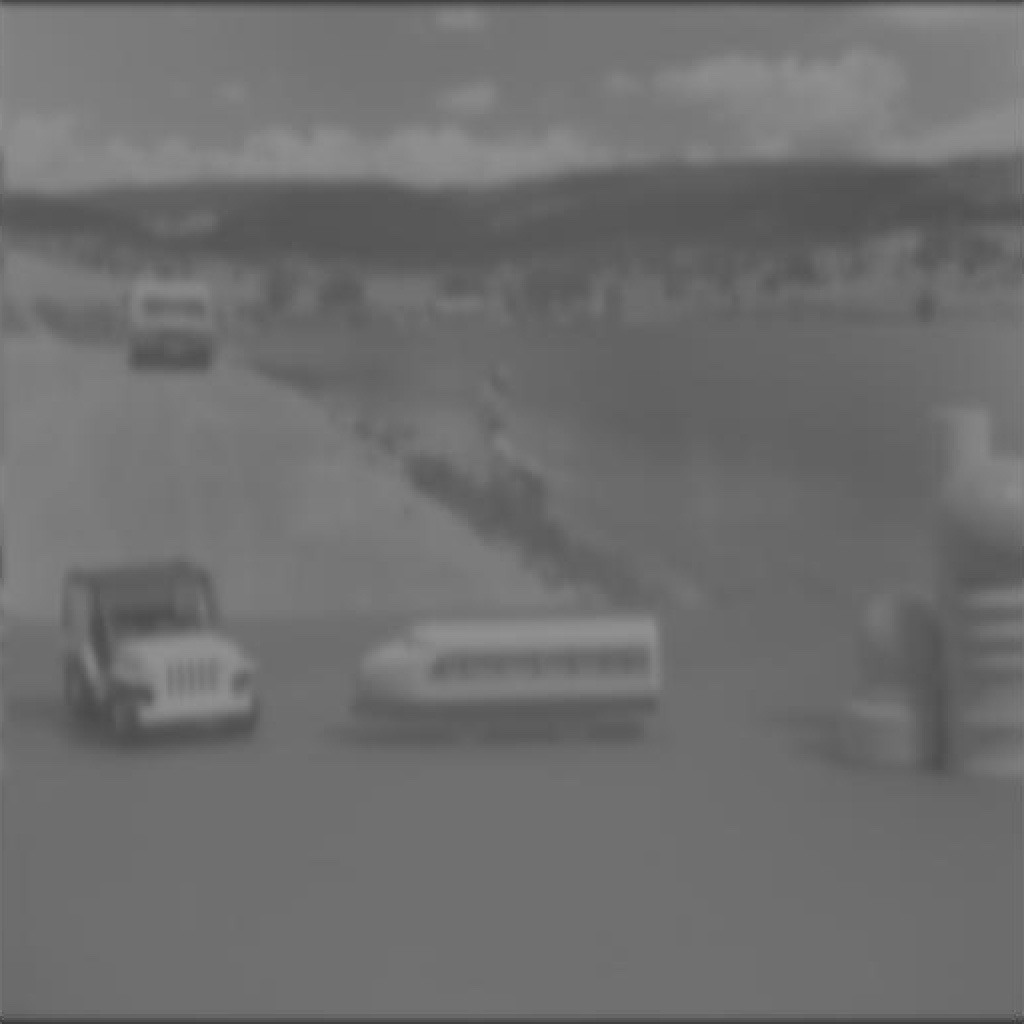}}
			\hspace{0.1cm}
						\\

	\subfigure{\includegraphics[width=0.71in]{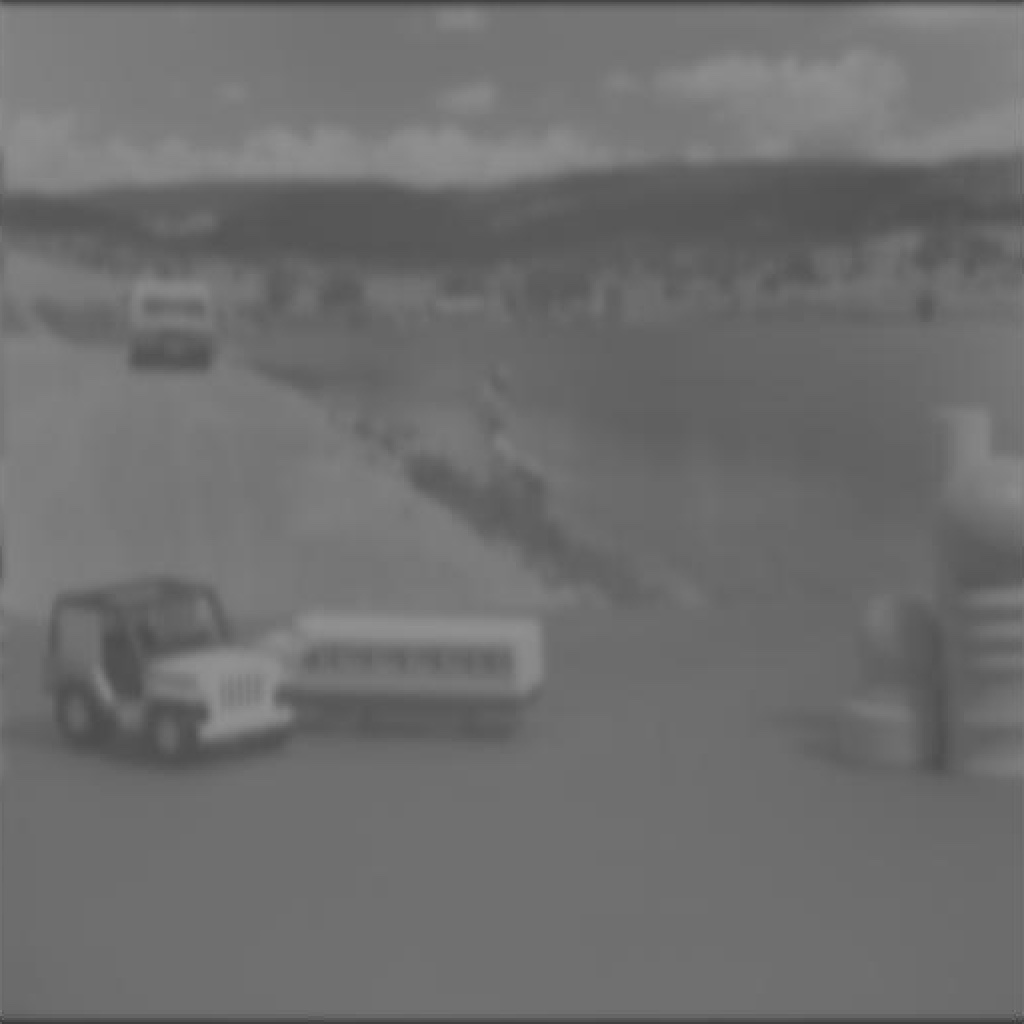}}
			\hspace{0.1cm}
	\subfigure{\includegraphics[width=0.71in]{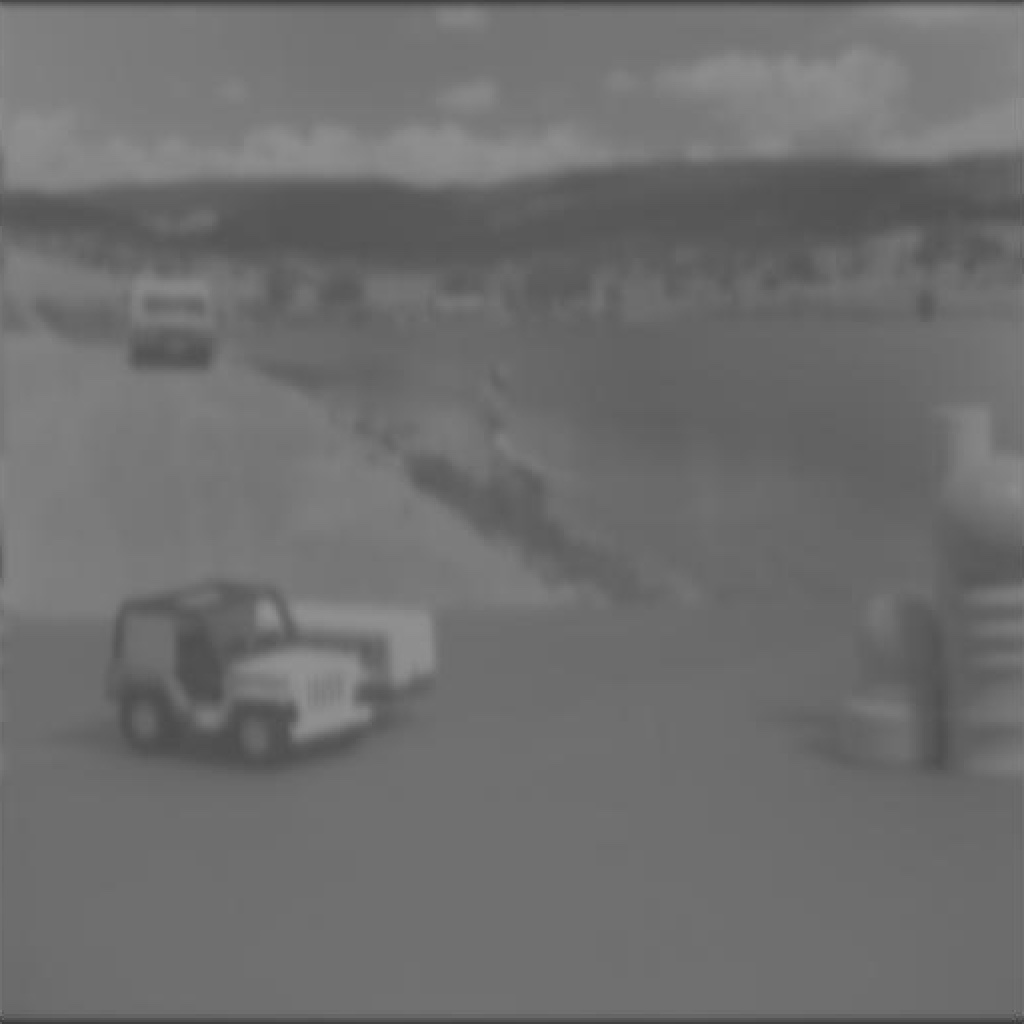}}
			\hspace{0.1cm}
	\subfigure{\includegraphics[width=0.71in]{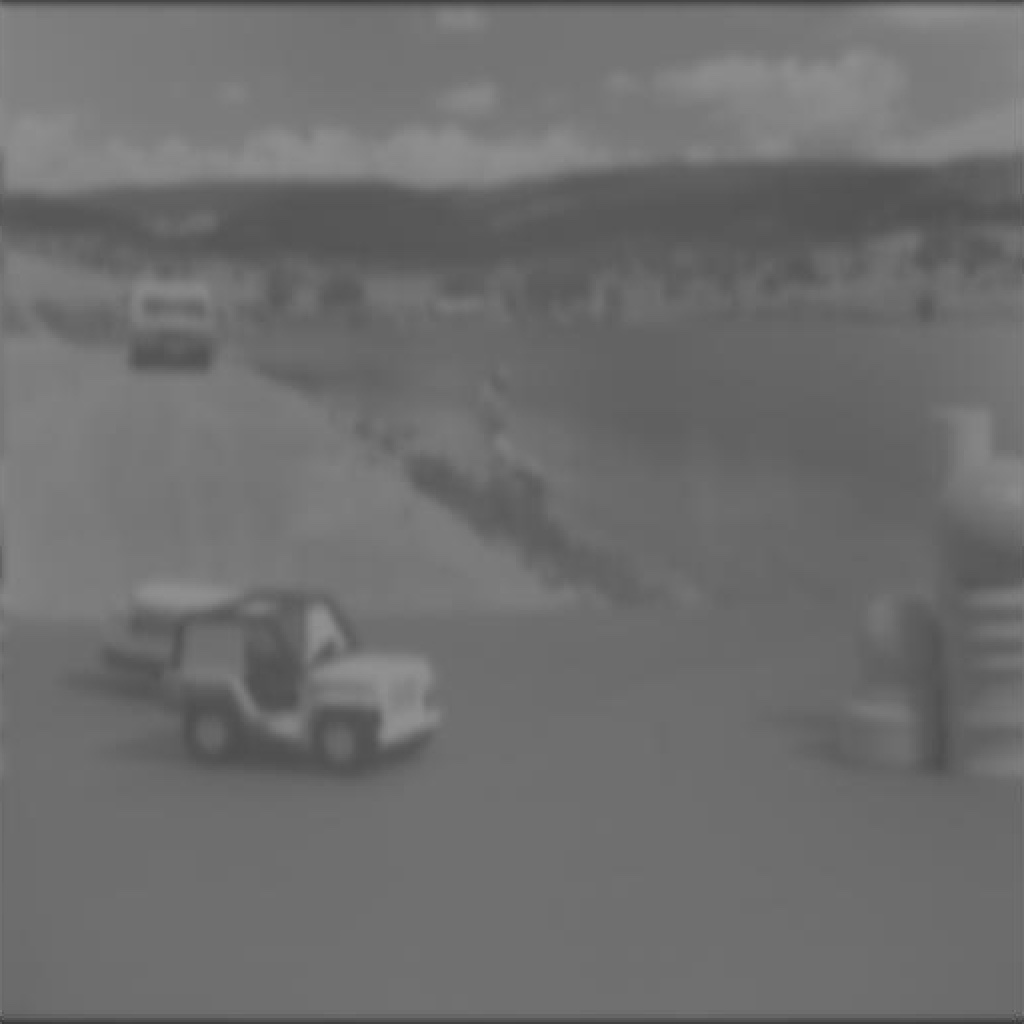}}
			\hspace{0.1cm}
	\subfigure{\includegraphics[width=0.71in]{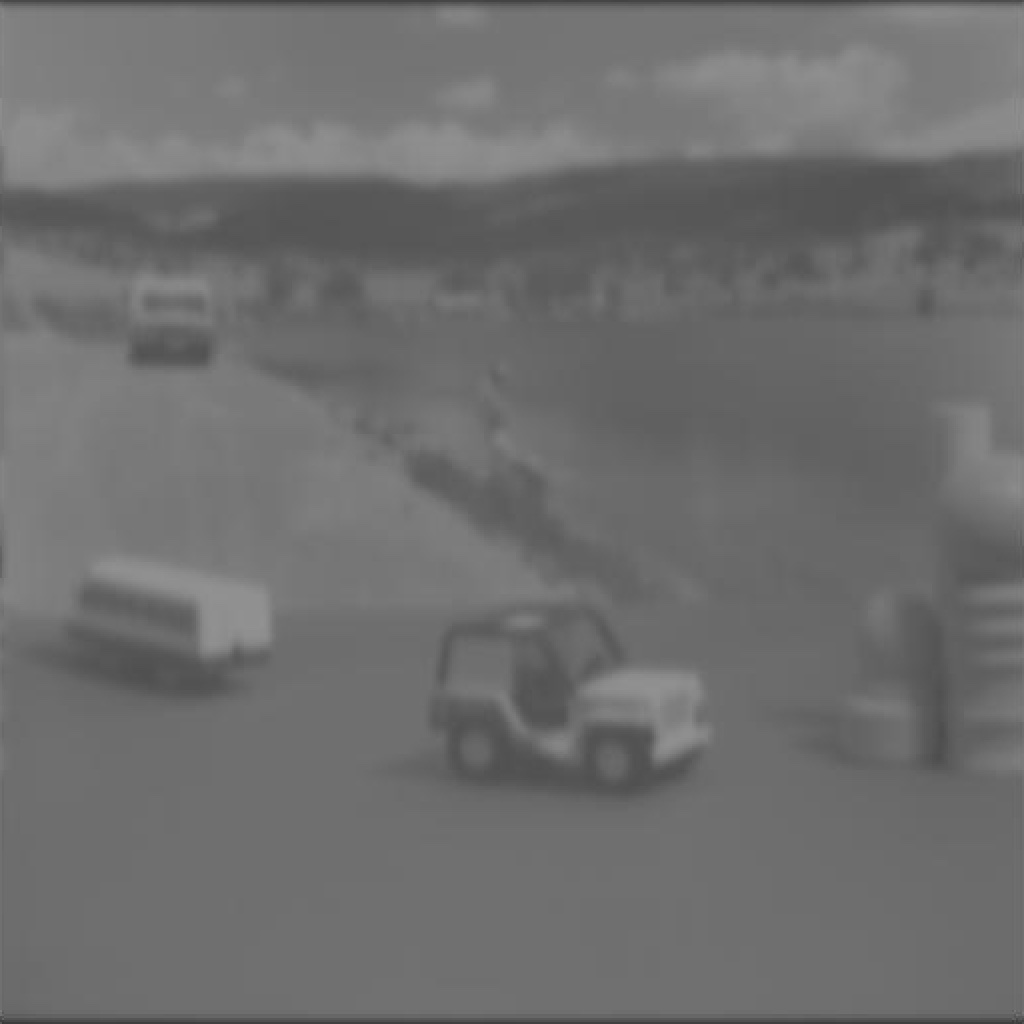}}
			\hspace{0.1cm}
			\subfigure{\includegraphics[width=0.71in]{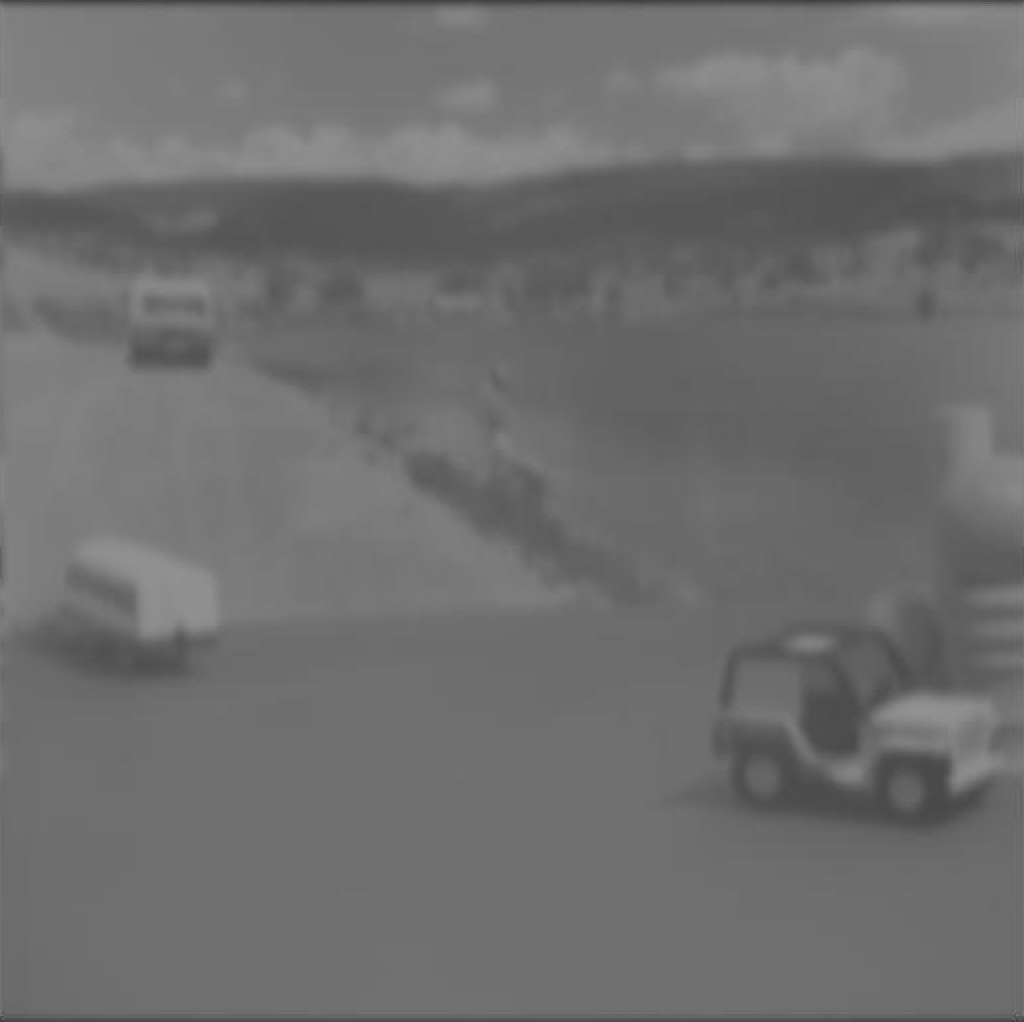}}
						\hspace{0.1cm}\\
\ 
	\subfigure{\includegraphics[width=0.71in]{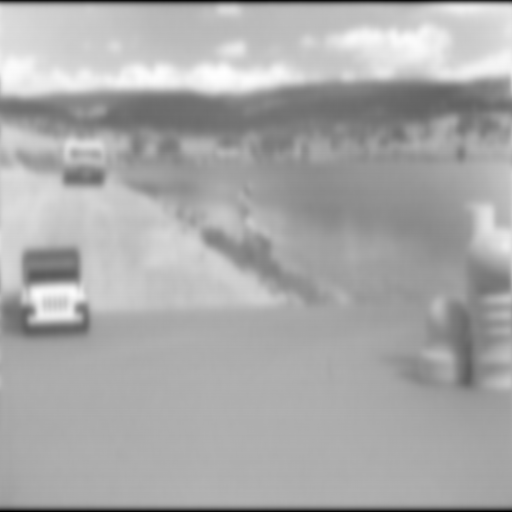}}
		\hspace{0.1cm}
	\subfigure{\includegraphics[width=0.71in]{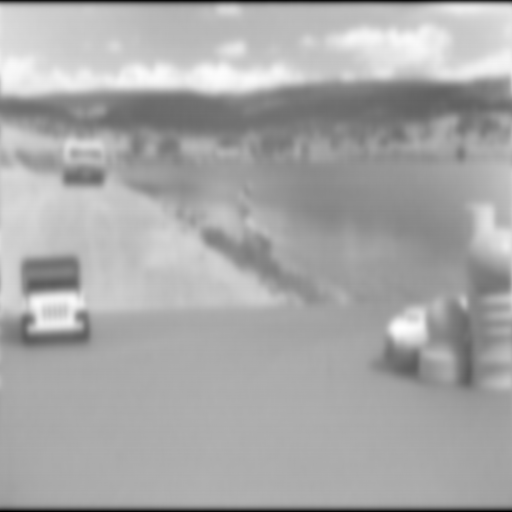}}
			\hspace{0.1cm}
	\subfigure{\includegraphics[width=0.71in]{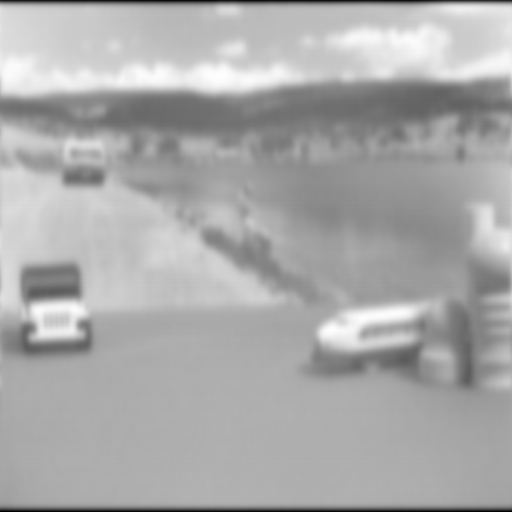}}
			\hspace{0.1cm}
	\subfigure{\includegraphics[width=0.71in]{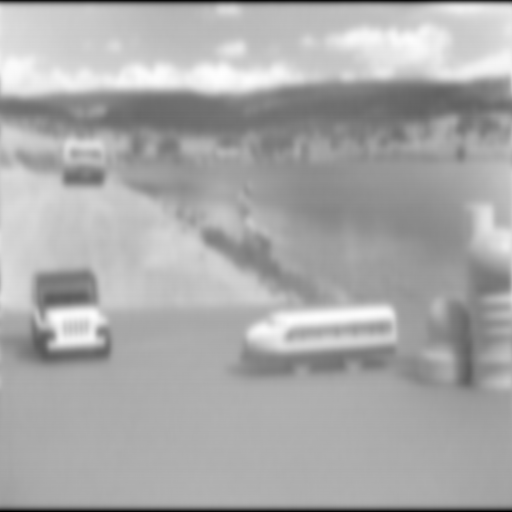}}
			\hspace{0.1cm}
	\subfigure{\includegraphics[width=0.71in]{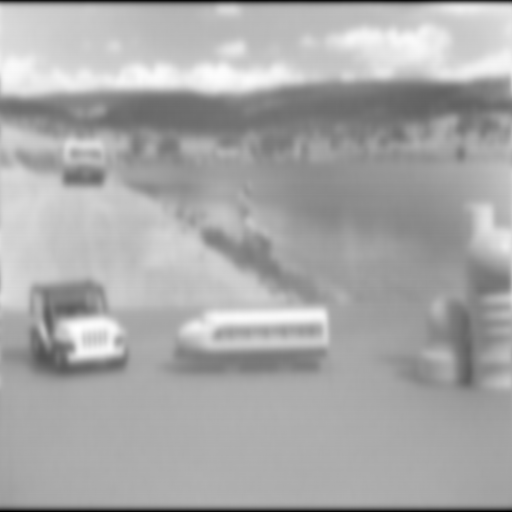}}
			\hspace{0.1cm}
			\\
\ \ 
	\subfigure{\includegraphics[width=0.71in]{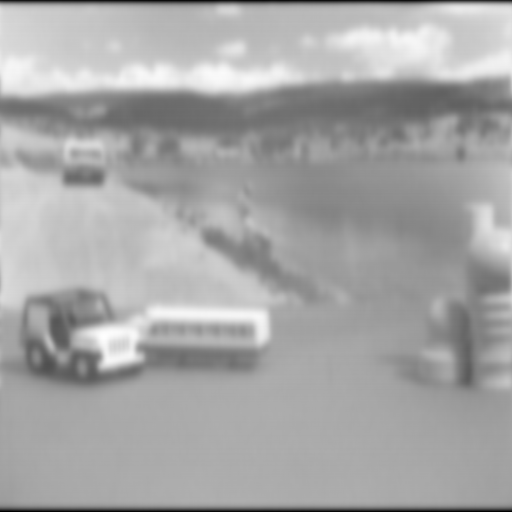}}
			\hspace{0.1cm}
	\subfigure{\includegraphics[width=0.71in]{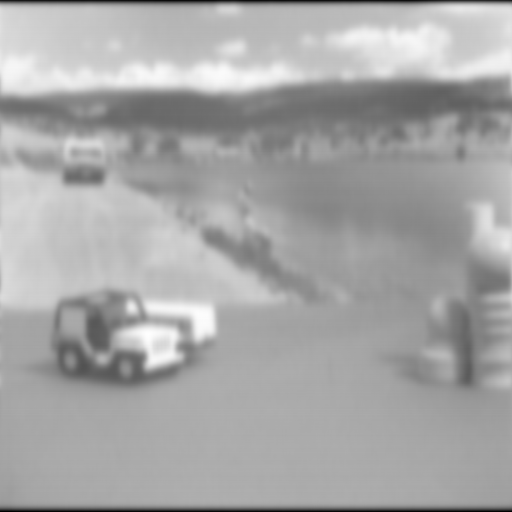}}
			\hspace{0.1cm}
	\subfigure{\includegraphics[width=0.71in]{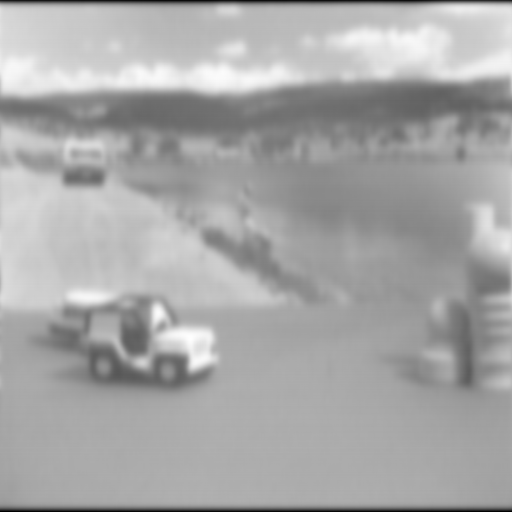}}
			\hspace{0.1cm}
	\subfigure{\includegraphics[width=0.71in]{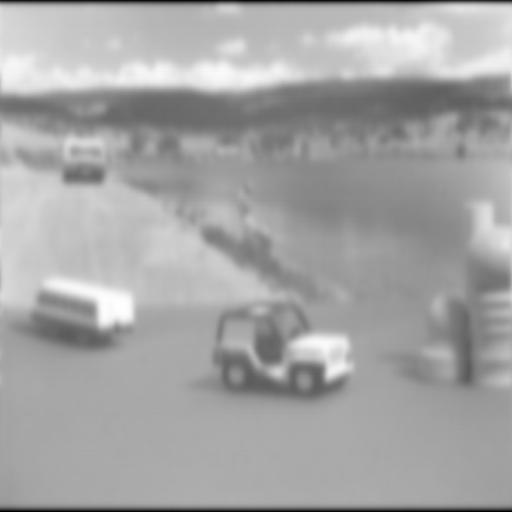}}
			\hspace{0.1cm}
	\subfigure{\includegraphics[width=0.71in]{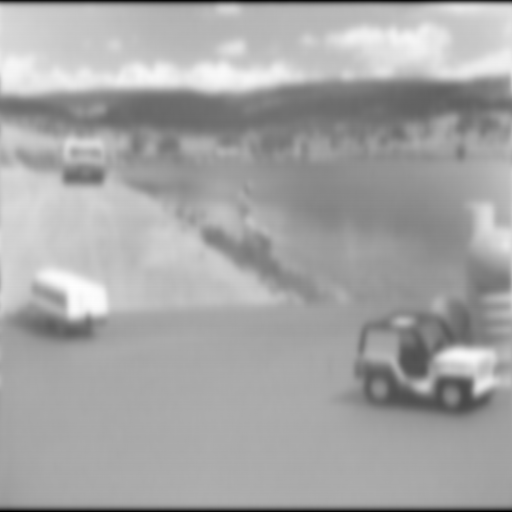}}
			\hspace{0.1cm}
	\caption{Original, Blurred and Deblurred Videos}
\end{figure}

Mean Square Error (MSE) is a common evaluation index in the field of image processing, 
\begin{equation}
{\rm MSE}(\mathcal X_1,\mathcal X_2)=\frac{\|\mathcal X_1-\mathcal X_2\|_F^2}{mnp},
\end{equation}
where $\mathcal X_1,\mathcal X_2\in \mathbb{R}^{m\times n\times p}$. MSE values before and after deblurring by the TR-TLS are compared with those of the original images as shown in Table \ref{MSE}.
\begin{table}[htbp]
\centering
\caption{\rm {MSE} Results of Deblurring by TR-TLS Method }\label{MSE}
\begin{tabular}{ccccc}
\toprule
\textbf{Data Category} & \textbf{Data Name} & \textbf{Blurred MSE} & \textbf{Deblurred MSE}  & \textbf{Restoring Proportion}\\
\hline
Gray Image     &     City  &     0.0596   &    0.0056   & 90.6\%\\

 Gray Image     &      Artificial Satellite              & 0.0525 & 0.0064 &87.8\%\\
 
  Color Image   &        Pepper            &   0.0566 &   0.0057 &89.9\%\\
  Color Image & Lena  & 0.0573& 0.0061&89.3\%\\
Video & Video & 0.1143  & 0.0121& 89.4\%\\
  \bottomrule
\end{tabular}
\end{table}

From all of these results, we can see that the TR-TLS method does make the effect of image and video deblurring. The deblurring operation by the TR-TLS can reduce the MSE index by about 90\%.
\newpage
\subsection{Comparisons between TR-TLS and Existing Methods}

In this subsection, we will take the experimental numerical data of color image Pepper as an example to show the advantages of the TR-TLS algorithm in terms of time cost and the MSE. In following figures, the horizontal axis is the MSE and the vertical axis is the CPU time of the program running.

We compare the TR-TLS with existing TTSVD \cite{beik2021tensor, fierro1997regularization}, RTTSVD \cite{xie2019randomized} and  TGGKB, TGGMRES \cite{reichel2022tensor} methods. Among them, 
the TTSVD is the abbreviation for truncated tensor SVD, while RTTSVD is generated from TTSVD, improved by randomized algorithms. Besides, both tensor global Golub-Kahan bidiagonalization algorithm (TGGKB) and generalized minimum residual method with tensor generalizations  (TGGMRES) are iterative algorithms to solve the TR-TLS problem, considering it as a tensor regularized LS problem and ignoring the errors in mapping tensor $\mathcal A$. For convenience, ``TGGKB\_5" is utilized to mean the experiment results after $5$ iteration steps, using the TGGKB algorithm.

 As the following two images show, the closer the data points are to the plane rectangular coordinate system, the lower the MSE is achieved at a smaller time cost in the numerical examples. The TR-TLS method has obvious advantages, which proves the correctness and effectiveness of our theory.

\begin{figure}[htbp]
\centering
\includegraphics[width=5in]{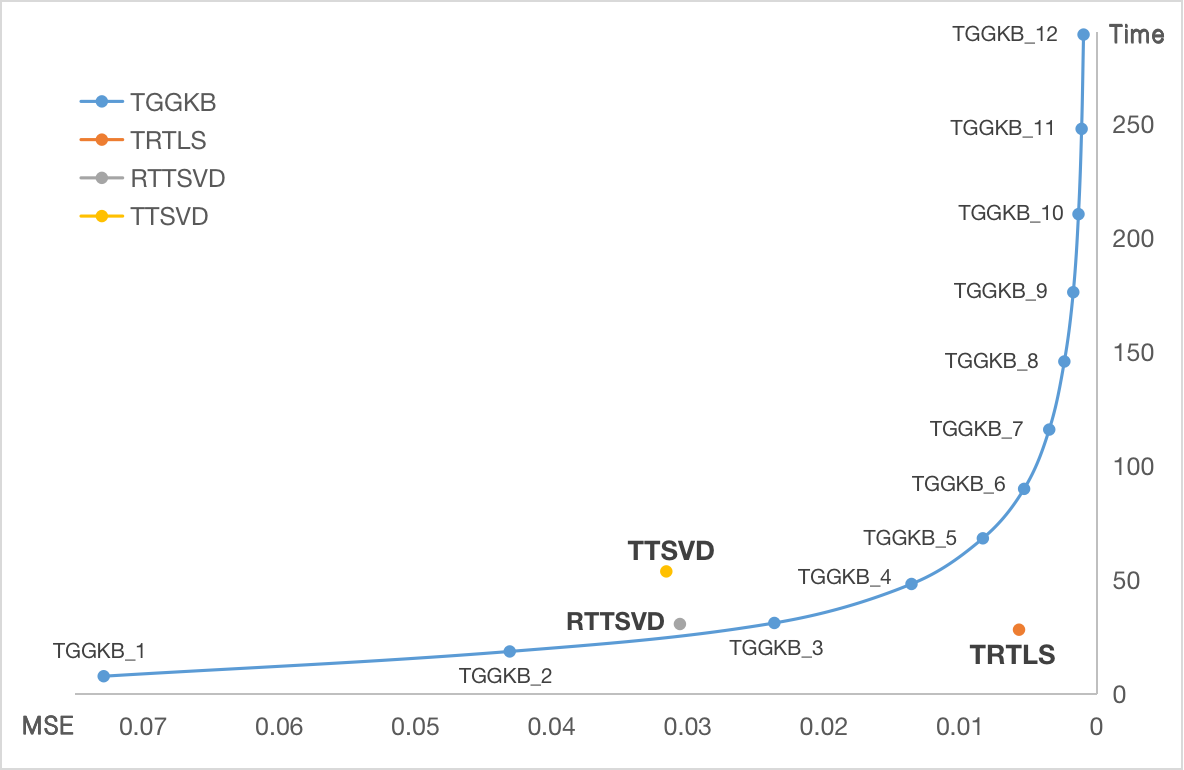}
\caption{Comparisons among TGGKB, RTTSVD, TTSVD and TR-TLS}
\end{figure}

\begin{figure}[htbp]
\centering
\includegraphics[width=5in]{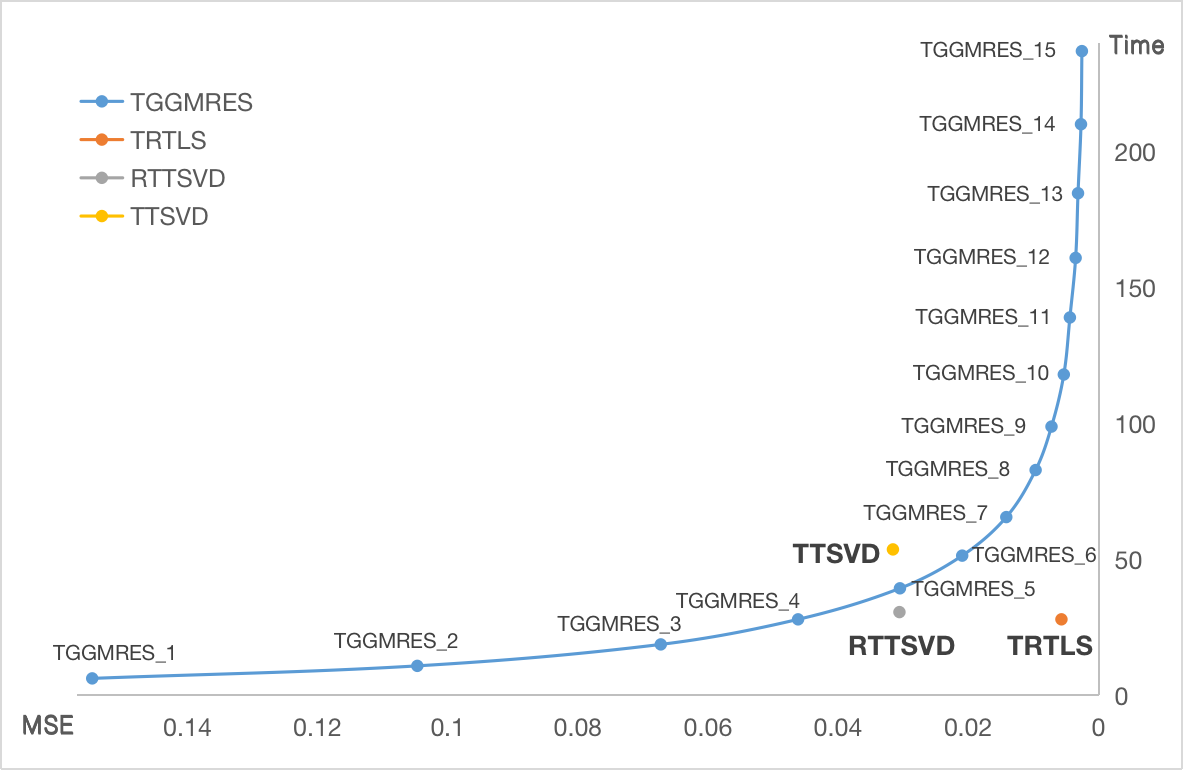}
\caption{Comparisons among TGGMRES, RTTSVD, TTSVD and TR-TLS}
\end{figure}

\section{Conclusions and Future Researches}

In this paper, the regularized TLS based on the tensor T-product is established. We extend the related theorems and properties of classical RTLS in the matrix form to the tensor form. Based on these theorems, numerical algorithms for the TR-TLS problem are proposed. In addition, we explore the applications of the TR-TLS method in the field of image and video deblurring. Through numerical experiments, the TR-TLS is proved to have obvious advantages over solving such ill-conditioned problems.

For further improving the iterative algorithms mentioned in this paper, we may consider both algorithmic adjustments and stability analysis. We would like to focus on a rigorous convergence theory. Another topic of investigation is to conduct perturbation analysis for the TR-TLS problem. Since the TLS is also known as the errors-in-variables model in the statistical literature, we expect that  new algorithms will show potential advantages for parameter estimation. Also, it is natural to consider a generalization of nonlinear TLS.

In addition to these mentioned above, future attention will also be paid to deeper mining of more randomized and preserving-structure algorithms, and wider applications of the TR-TLS in the real world.

\newpage

{\small
\bibliographystyle{siam}
\bibliography{document}
}

\end{document}